\documentclass[12pt]{article}
\usepackage{amsmath,amssymb,amsthm,amsfonts,mathrsfs}
\usepackage{appendix}
\usepackage{array,bm,enumitem,graphics,xcolor,yfonts,colonequals}
\usepackage{stmaryrd}
\usepackage[alphabetic]{amsrefs}
\usepackage[all,cmtip]{xy}
\usepackage{tikz-cd}
\usepackage[colorlinks,anchorcolor=blue,citecolor=blue,linkcolor=blue,urlcolor=blue,bookmarksopen=true]{hyperref}
\usepackage{comment}
\usepackage{mathtools}
\usepackage[margin=1in]{geometry}

\usepackage{titlesec}
\titleformat{\section}
  {\normalfont\large\bf}{\thesection}{1em}{}
\titleformat{\subsection}[runin]
  {\normalfont\normalsize\bf}{\thesubsection}{1em}{}

\usepackage{fancyhdr}
\fancypagestyle{titlepage}
{
	\fancyhf{}

	\fancyfoot[l]{
	\href{https://mathscinet.ams.org/mathscinet/msc/msc2020.html}
	}
}

\AtBeginDocument{%
	\def\MR#1{}
}

\newcommand{\bC}{\mathbb{C}}    

\newcommand{\bP}{\mathbb{P}}    
\newcommand{\bQ}{\mathbb{Q}}    
\newcommand{\bR}{\mathbb{R}}    
\newcommand{\bZ}{\mathbb{Z}}    

\newcommand{\cB}{\mathcal{B}}   
\newcommand{\cD}{\mathcal{D}}   
\newcommand{\cH}{\mathcal{H}}   
\newcommand{\cK}{\mathcal{K}}   
\newcommand{\cM}{\mathcal{M}}   

\newcommand{\cO}{\mathcal{O}}   
\newcommand{\cQ}{\mathcal{Q}}   
\newcommand{\cX}{\mathcal{X}}

\newcommand{\sD}{\mathscr{D}}   
\newcommand{\sP}{\mathscr{P}}   

\newcommand{\Amp}{\mathrm{Amp}}         

\newcommand{\cpx}{\mathrm{cpx}}
\newcommand{\CY}{\mathrm{CY}}
\newcommand{\Db}{\mathrm{D^b}}  
\newcommand{\GL}{\mathrm{GL}}   
\newcommand{\Gr}{\mathrm{Gr}}   
\newcommand{\Kah}{\mathrm{K\ddot{a}h}}  
\newcommand{\uM}{\mathrm{M}}   
\newcommand{\uO}{\mathrm{O}}            
\newcommand{\PSL}{\mathrm{PSL}}         
\newcommand{\rank}{\mathrm{rank}}
\newcommand{\red}{\mathrm{red}}
\newcommand{\SO}{\mathrm{SO}}           
\newcommand{\Spec}{\mathrm{Spec}}   
\newcommand{\Stab}{\mathrm{Stab}} 
\newcommand{\vol}{\mathrm{vol}}

\newcommand{\Aut}{\operatorname{Aut}}   
\newcommand{\Hom}{\operatorname{Hom}}   

\newcommand{\NS}{\operatorname{NS}}     

\newcommand{\txtwedge}{{\textstyle\bigwedge}}

\newtheorem*{thm*}{Theorem}
\newtheorem*{prop*}{Proposition}
\newtheorem*{cor*}{Corollary}
\newtheorem{thm}{Theorem}[section]
\newtheorem{prop}[thm]{Proposition}
\newtheorem{cor}[thm]{Corollary}
\newtheorem{lem}[thm]{Lemma}

\numberwithin{equation}{section}

\theoremstyle{definition}
\newtheorem{defn}[thm]{Definition}

\newtheorem{deflem}[thm]{Definition-Lemma}
\newtheorem{eg}[thm]{Example}
\newtheorem{ques}[thm]{Question}
\newtheorem{rmk}[thm]{Remark}

\begin{document}
\title{Mirror symmetry for lattice-polarized abelian surfaces}
\author{Yu-Wei Fan
    \and Kuan-Wen Lai}
\date{}

\newcommand{\ContactInfo}{{
\bigskip\footnotesize

\bigskip
\noindent Y.-W.~Fan,
\textsc{Center for Mathematics and Interdisciplinary Sciences, Fudan University \\
Shanghai 200433, China}
\par\nopagebreak
\smallskip
\noindent\textsc{Shanghai Institute for Mathematics and Interdisciplinary Sciences (SIMIS) \\ 
Shanghai 200433, China}
\par\nopagebreak
\noindent\textsc{Email:} \texttt{yuweifanx@gmail.com}

\bigskip
\noindent K.-W.~Lai,
\textsc{Department of Smart Computing and Applied Mathematics \\
Tunghai University \\
No.~1727, Sec.~4, Taiwan~Blvd., Xitun~Dist., Taichung~City 407224, Taiwan}
\par\nopagebreak
\smallskip
\noindent\textsc{National Center for Theoretical Science \\
No.~1, Sec.~4, Roosevelt~Rd., Taipei~City 106319, Taiwan}
\par\nopagebreak
\noindent\textsc{Email:} \texttt{kwlai@thu.edu.tw}
}}

\maketitle
\thispagestyle{titlepage}

\begin{abstract}
Inspired by the Dolgachev--Nikulin--Pinkham mirror symmetry for lattice-polarized K3 surfaces, we study its analogue for abelian surfaces. In this paper, we introduce lattice-polarized abelian surfaces and construct their coarse moduli spaces. We then construct stringy K\"ahler moduli spaces for abelian surfaces and show that these two spaces are naturally identified for mirror pairs. We also introduce a natural involution on stringy K\"ahler moduli spaces which, under mirror symmetry, pairs abelian surfaces and their duals. Finally, we determine conditions for the existence of mirror partners and classify self-mirror abelian surfaces via their N\'eron--Severi lattices.
\end{abstract}

\setcounter{tocdepth}{3}
\tableofcontents

\section{Introduction}

Mirror symmetry predicts deep relationships between the complex geometry of a Calabi--Yau manifold and the symplectic geometry of its mirror, which may be formulated as a correspondence between their complex and stringy K\"ahler moduli spaces. Although a general mathematical definition of the stringy K\"ahler moduli space remains elusive, it can be rigorously formulated for curves and surfaces through Bridgeland stability conditions. In fact, the attempt to construct such a moduli space was among the original motivations for introducing stability conditions on triangulated categories; see \cite{Bri07}*{Section~1.4} and \cite{Bri09Survey}*{Section~2.2}.

For elliptic curves, Bridgeland \cite{Bri09Survey}*{Section~4.1} showed that the manifold of stability conditions provides a natural model for the stringy K\"ahler moduli space. Since an elliptic curve is self-mirror, this space coincides with the complex moduli space, both of which can be identified with $\cH/\PSL(2,\bZ)$, where $\cH\subseteq\bC$ is the upper half-plane. In dimension two, mirror symmetry for lattice-polarized K3~surfaces was originally introduced by Dolgachev \cites{Dol83,Dol96}, Nikulin \cite{Nik79_K3}, and Pinkham \cite{Pin77}. This was later formulated in terms of stability conditions by Bayer and Bridgeland \cite{BB17}*{Section~7}, who constructed stringy K\"ahler moduli spaces in a similar way and showed that they are isomorphic to coarse moduli spaces of lattice-polarized K3~surfaces.

In this paper, we extend this picture to abelian surfaces, starting with the notion of lattice-polarized abelian surfaces and the construction of their moduli spaces. For K3~surfaces, this construction is based on the notion of marked K3~surfaces, namely K3~surfaces~$X$ equipped with lattice isometries
$$\xymatrix{
    H^2(X,\bZ)
    \ar[r]^-\sim
    & E_8(-1)^{\oplus 2}\oplus U^{\oplus 3},
}$$
where $E_8$ is the unique even unimodular lattice of rank~$8$ and
$$U = \begin{pmatrix}
    0 & 1 \\
    1 & 0
\end{pmatrix}$$
is the hyperbolic lattice of rank~$2$. For abelian surfaces, by contrast, the marking is placed on the \emph{first cohomology group}. This reflects the fact that, unlike K3~surfaces, the period of an abelian surface does not uniquely determine the surface itself \cite{Shioda}*{Theorem~II}.

To formulate mirror symmetry for abelian surfaces, we consider two lattices associated with an abelian surface $X$. Let $\NS(X)\subseteq H^2(X,\bZ)$ denote the N\'eron--Severi lattice, and denote the \emph{transcendental lattice} by
$$
    T(X)\colonequals
    \NS(X)^\perp\subseteq H^2(X,\bZ).
$$
The even cohomology group
$$
    H^{\rm ev}(X, \bZ)\colonequals
    H^0(X,\bZ)\oplus H^2(X,\bZ)\oplus H^4(X,\bZ)
$$
carries a natural bilinear form, known as the \emph{Mukai pairing}, which extends the intersection pairing on $H^2(X, \bZ)$ via
$$
    (r, D, s)\cdot (r', D', s')
    = DD' - rs' - sr'.
$$
The \emph{numerical Grothendieck group} of the derived category $\Db(X)$ is then given by
$$
    N(X)\colonequals
    H^0(X,\bZ)\oplus\NS(X)\oplus H^4(X,\bZ)
$$
which becomes a lattice under the Mukai pairing. Note that $N(X)\cong\NS(X)\oplus U$.

\begin{defn}
Two abelian surfaces $X$ and $Y$ are called \emph{mirror partners} if there exist lattice isometries $T(X)\cong N(Y)$ and $N(X)\cong T(Y)$.
\end{defn}

In fact, the existence of an isometry $T(X)\cong N(Y)$ guarantees $N(X)\cong T(Y)$, and vice versa; see Corollary~\ref{cor:numGroth-trans}.

\begin{thm}
\label{thm:mirror-moduli}
For an abelian surface $X$ with a mirror partner $Y$, one can associate two moduli spaces:
\begin{itemize}
    \item the coarse moduli space of $\NS(X)$-polarized abelian surfaces $\cM_\cpx(X)$, and
    \item the stringy K\"ahler moduli space $\cM_\Kah(X)$,
\end{itemize}
such that there are natural isomorphisms
$$
    \cM_\cpx(X)\cong\cM_\Kah(Y)
    \qquad\text{and}\qquad
    \cM_\Kah(X)\cong\cM_\cpx(Y).
$$
\end{thm}

The stringy K\"ahler moduli space of an abelian surface~$X$ can be realized as the quotient of the space of \emph{complexified K\"ahler classes}
$$
    \cK(X)\colonequals\left\{
        \omega\in\NS(X)\otimes\bC
    \;\middle|\;
        \mathrm{Im}(\omega)\in\Amp(X)
    \right\}
$$
by the action of an arithmetic group. We show that $\cK(X)$ admits a naturally defined involution that descends to an involution on $\cM_\Kah(X)$. Under mirror symmetry, this involution corresponds precisely to taking dual abelian surfaces on the complex moduli.

\begin{thm}
\label{thm:symp-dual}
For an abelian surface $X$, the map
$$
    \iota\colon\cK(X)\longrightarrow\cK(X)
    :\omega\longmapsto\frac{\omega}{-\frac{1}{2}\omega^2}
$$
defines an involution that inverts the K\"ahler volume $\vol(\omega)\colonequals -\frac{1}{2}\omega^2$, meaning that
$$
    \vol(\iota(\omega)) = \vol(\omega)^{-1}.
$$
Moreover, it satisfies the following properties:
\begin{enumerate}[label=\textup{(\alph*)}]
\item\label{symp-dual:invol-str-Kah}
It descends to an involution
$
    \sD\colon\cM_\Kah(X)\longrightarrow\cM_\Kah(X).
$
\item\label{symp-dual:mirror-dual}
Suppose $X$ admits a mirror partner $Y$. Fix an isometry $T(Y)\cong N(X)$ and let
$$
    \mathscr{M}\colon\cM_\cpx(Y)\longrightarrow\cM_{\Kah}(X)
$$
be the induced mirror isomorphism. Then the composition
$$
    \mathscr{M}^{-1}\circ\sD\circ\mathscr{M}
    \colon\cM_\cpx(Y)\longrightarrow\cM_\cpx(Y)
$$
sends an $\NS(Y)$-polarized abelian surface to an $\NS(Y)$-polarized abelian surface such that the underlying abelian surfaces are dual to each other.
\end{enumerate}
\end{thm}

In light of this theorem, we define:

\begin{defn}
Let $X$ be an abelian surface. For every complexified K\"ahler class $\omega\in\cK(X)$, we call the class $-2\omega/\omega^2\in\cK(X)$ the \emph{symplectic dual} of $\omega$.
\end{defn}

\begin{ques}
In complex geometry, an abelian variety $A$ and its dual $\widehat{A}$ are derived equivalent, that is, $\Db(A)\cong\Db(\widehat{A})$, via the Fourier--Mukai transform induced by the Poincar\'e bundle on $A\times\widehat{A}$ \cite{Muk81}. On the symplectic side, are the Fukaya categories of $X$ with respect to $\omega$ and $-2\omega/\omega^2$ related through certain Lagrangian correspondences?
\end{ques}

The following result characterizes when an abelian surface admits a mirror partner:

\begin{thm}
\label{thm:crit_mirror-partner}
An abelian surface $X$ with Picard number $\rho(X)$ admits a mirror partner if and only if
\begin{itemize}
    \item $\rho(X)\leq2$, or
    \item $\rho(X)=3$ and $\NS(X)\cong\bZ(-2n)\oplus U$ for some positive integer~$n$.
\end{itemize}
\end{thm}

For a pair of mirror partners, either both surfaces have Picard number~$2$, or one surface has Picard number~$1$ and the other has Picard number~$3$. In the latter case, each surface uniquely determines the N\'eron--Severi group (and hence the numerical Grothendieck group and transcendental lattice) of its mirror, as shown in the following table:

\begin{center}
\renewcommand{\arraystretch}{1.2}
\begin{tabular}{|c||c|c|c|}
\hline
    Mirror partners
    & $\NS(\;\cdot\;)$
    & $N(\;\cdot\;)$
    & $T(\;\cdot\;)$ \\
\hline
\hline
    $X$
    & $\bZ(2n)$
    & $\bZ(2n)\oplus U$
    & $\bZ(-2n)\oplus U^{2}$ \\
\hline
    $Y$
    & $\bZ(-2n)\oplus U$
    & $\bZ(-2n)\oplus U^{2}$
    & $\bZ(2n)\oplus U$ \\
\hline
\end{tabular}
\end{center}

The case of Picard number~$2$ is more subtle. In this setting, abelian surfaces with non-isomorphic N\'eron--Severi lattices may share the same mirror. In fact, two abelian surfaces have the same mirror if and only if their N\'eron--Severi lattices are stably equivalent, or equivalently, have isomorphic discriminant forms; see Proposition~\ref{prop:stably-equiv}. Moreover, an abelian surface $X$ of Picard number~$2$ may be \emph{self-mirror}, namely $T(X)\cong N(X)$. A basic example is the product of two non-isogenous elliptic curves. The following result gives a criterion for identifying self-mirror abelian surfaces in terms of discriminant forms.

\begin{thm}
\label{thm:self-mirror}
Let $X$ be an abelian surface with Picard number $2$. Then $X$ is self-mirror if and only if the discriminant form $(A,q)$ of $\NS(X)$ satisfies the following conditions:
\begin{enumerate}[label=\textup{(\alph*)}]
    \item\label{self-mirror:odd-p}
    For any prime $p\equiv3\pmod{4}$, the $p$-Sylow subgroup $A_p\subseteq A$ has the form
    $$
        A_p\cong(\bZ/p^k\bZ)\oplus(\bZ/p^k\bZ)
    $$
    for some integer $k$.
    \item The $2$-Sylow subgroup $A_2\subseteq A$ equipped with the restricted form $q_2 \colonequals q\mid_{A_2}$ is isomorphic to one of the following:
    \begin{enumerate}[label=\textup{(\roman*)}]
        \item $A_2\cong(\bZ/2^{k}\bZ)\oplus(\bZ/2^{k}\bZ)$ for some integer $k$, and there exist odd integers $\theta_1, \theta_2$ with $\theta_1 + \theta_2\equiv 0\pmod{4}$ such that
        $$
            q_2(\overline{a}_1,\overline{a}_2)
            = \frac{\theta_1 a_1^2}{2^{k}}
                + \frac{\theta_2 a_2^2}{2^{k}}
            \qquad\text{for each}\qquad
            (\overline{a}_1, \overline{a}_2)\in A_2.
        $$
        Here $a_1, a_2$ are integers representing $\overline{a}_1, \overline{a}_2\in\bZ/2^{k}\bZ$.
        \item $A_2\cong(\bZ/2^{k}\bZ)\oplus(\bZ/2^{k+1}\bZ)$ for some integer $k$, and there exist odd integers $\theta_1, \theta_2$ with $\theta_1 + \theta_2\equiv 0\pmod{4}$ such that
        $$
            q_2(\overline{a}_1,\overline{a}_2)
            = \frac{\theta_1 a_1^2}{2^{k}}
                + \frac{\theta_2 a_2^2}{2^{k+1}}
             \qquad\text{for each}\qquad
            (\overline{a}_1, \overline{a}_2)\in A_2.
        $$
        Here $a_1, a_2$ are integers representing $\overline{a}_1\in\bZ/2^{k}\bZ$ and $\overline{a}_2\in\bZ/2^{k+1}\bZ$.
        \item $A_2\cong(\bZ/2^{k}\bZ)\oplus(\bZ/2^{k}\bZ)$ for some integer $k$, and there exists an odd integer $\theta$ such that
        $$
            q_2(\overline{a}_1,\overline{a}_2)
            = \frac{\theta a_1a_2}{2^{k-1}}
            \qquad\text{for each}\qquad
            (\overline{a}_1, \overline{a}_2)\in A_2.
        $$
        Here $a_1, a_2$ are integers representing $\overline{a}_1, \overline{a}_2\in\bZ/2^{k}\bZ$.
        \item $A_2\cong(\bZ/2^{k}\bZ)\oplus(\bZ/2^{k}\bZ)$ for some integer $k$, and there exists an odd integer $\theta$ such that
        $$
            q_2(\overline{a}_1,\overline{a}_2)
            = \frac{\theta(a_1^2+a_1a_2+a_2^2)}{2^{k-1}}
             \qquad\text{for each}\qquad
            (\overline{a}_1, \overline{a}_2)\in A_2.
        $$
        Here $a_1, a_2$ are integers representing $\overline{a}_1, \overline{a}_2\in\bZ/2^{k}\bZ$.
    \end{enumerate}
\end{enumerate}
\end{thm}

For principally polarized abelian surfaces, the conditions for being self-mirror reduce to a particularly simple form:

\begin{cor}
\label{cor:self-mirror_prinAbS}
A principally polarized abelian surface $X$ of Picard number~$2$ is self-mirror if and only if the discriminant of $\NS(X)$ is not divisible by $16$ nor any prime $p\equiv3\pmod{4}$.
\end{cor}

\bigskip
\noindent\textbf{Organization of the paper.}
In Section~\ref{sec:lattice-pol-AbS}, we define lattice-polarized abelian surfaces and construct their coarse moduli spaces. Section~\ref{sec:mirror-pair} is devoted to the construction of stringy K\"ahler moduli spaces, as well as proving Theorems~\ref{thm:mirror-moduli}, \ref{thm:symp-dual}, and \ref{thm:crit_mirror-partner}. In Section~\ref{sec:self-mirror}, we prove Theorem~\ref{thm:self-mirror}, Corollary~\ref{cor:self-mirror_prinAbS}, and present examples of self-mirror abelian surfaces.

\bigskip
\noindent\textbf{Acknowledgment.}
We are grateful to Igor Dolgachev for clarifying the correct sources for the notions of lattice-polarized K3~surfaces and their mirror symmetry. The first author is partially supported by the National Natural Science Foundation of China (Grant No.~12401053). The second author is supported by the NSTC Research Grant 113-2115-M-029-003-MY3 and partially supported by the Academic Summit Program 114-2639-M-002-009-ASP. The second author also thanks Pedro N\'u\~nez and Flora Poon for organizing the conference \emph{Tour de Formosa} (January~12--16, 2026), during which most participants cycled through eastern and southern parts of Taiwan. Although the second author traveled by train, he sincerely appreciates the opportunity to present this work and to bring along his 17-month-old child for the little one’s very first journey.

\section{Lattice-polarized abelian surfaces}
\label{sec:lattice-pol-AbS}

In this section, we introduce the notion of lattice-polarized abelian surfaces and construct their coarse moduli spaces. We begin with a review of marked complex two-tori and set up the notation to be used throughout the paper.

\subsection{Marked complex 2-tori}
\label{subsec:marked-cplx-tori}

The main references for the following material are  \cite{BLComplexTori}*{Sections~1.10, 7.3} and \cite{Shioda}. Recall that for a complex torus $X$ of dimension~$2$,
$$
    H^1(X,\bZ)\cong\bZ^4, \qquad
    H^2(X,\bZ)\cong\txtwedge^2 H^1(X,\bZ)\cong\bZ^6,
$$
and under the cup product
$$
    H^2(X,\bZ)\times H^2(X,\bZ)
    \longrightarrow H^4(X,\bZ)\cong\bZ,
$$
the group $H^2(X,\bZ)$ has the structure of a lattice isometric to $U^{\oplus3}$. In view of these facts, throughout the paper we adopt the following conventions:
\begin{itemize}
    \item Let $L\cong\bZ^4$ be the free $\bZ$-module of rank $4$, and fix an isomorphism $\txtwedge^4L\cong\bZ$.
    \item Let $\Lambda\coloneqq\txtwedge^2L\cong U^{\oplus 3}$ be the rank-$6$ lattice endowed with the bilinear form
    $$
        \txtwedge^2L\times\txtwedge^2L
        \longrightarrow
        \txtwedge^4L\cong\bZ
        : (v, v')\longmapsto v\wedge v'.
    $$
\end{itemize}

Following \cite{Shioda}*{Section~1}, we call an ordered basis $\left(u^1,u^2,u^3,u^4\right)$ of $L$ \emph{admissible} if
$$
    u^1\wedge u^2\wedge u^3\wedge u^4
    = 1
$$
under the chosen isomorphism $\txtwedge^4L\cong\bZ$. Note that for $L = H^1(X,\bZ)$, the admissibleness is naturally defined since the orientation of $X$ determines a canonical isomorphism
$$
    \txtwedge^4L\cong H^4(X,\bZ)\cong\bZ.
$$
A \emph{marking} of $X$ is a choice of a basis of $H^1(X,\bZ)$, that is, an isomorphism
$$
    \mu\colon H^1(X,\bZ)\longrightarrow L.
$$
We call a marking \emph{admissible} if it maps an admissible basis to an admissible one, which holds precisely when the induced map
$$
    \wedge^2\mu\colon H^2(X,\bZ)\longrightarrow\Lambda
$$
is an isometry of lattices.

More generally, consider a family $\pi\colon\cX\rightarrow S$, namely a flat holomorphic map of complex manifolds, such that the fibers are complex $2$-tori. A \emph{marking} of such a family is an isomorphism of local systems
$$
    \mu\colon R^1\pi_*\bZ
    \longrightarrow
    \underline{L}_S.
$$
The moduli functor for marked complex tori $(\cX, \mu)$ admits a fine moduli space, constructed as follows. We identify the Grassmannian $\Gr(2, 4)$ with the space of quotient maps
$$
    w\colon\bC^{4}\longrightarrow V\cong\bC^2.
$$
After choosing a basis for $V$, each element $w\in\Gr(2, 4)$ is represented by a matrix
$$
    \Pi = \begin{pmatrix}
        \Pi_{11} & \Pi_{12} & \Pi_{13} & \Pi_{14} \\
        \Pi_{21} & \Pi_{22} & \Pi_{23} & \Pi_{24}
    \end{pmatrix}\in\uM_{2\times4}(\bC).
$$
One checks that the quotient $V/w(\bZ^{4})$ is a complex torus precisely when
$$
    \det\begin{pmatrix}
        \Pi \\
        \overline{\Pi}
    \end{pmatrix} \neq 0.
$$
This condition is equivalent to the kernels of $w$ and $\overline{w}$ intersecting transversely, and may therefore be abbreviated as
$
    w\wedge\overline{w} \neq 0.
$
By \cite{BLComplexTori}*{Theorem~7.3.1}, the open subset
$$
    \mathcal{B}\colonequals\{
        w\in\Gr(2, 4)
    \mid
        w\wedge\overline{w}\neq 0
    \}
$$
is a fine moduli space of marked complex $2$-tori. Writing $\{e_1, e_2, e_3 ,e_4\}$ for the standard basis of $\bZ^{4}\subseteq\bC^{4}$, each element $w\colon\bC^{4}\rightarrow V$ in $\cB$ corresponds to a complex torus $V/w(\bZ^4)$, marked by the dual basis of $\{w(e_1), w(e_2), w(e_3), w(e_4)\}$.

The \emph{period} of a marked complex $2$-torus
$$
    (X,\; \mu\colon H^1(X,\bZ)\rightarrow L)
    \in\cB,
$$
is the point in $\bP(\Lambda\otimes\bC)$ that represents the image of $H^{2,0}(X)\subseteq H^2(X,\bC)$ under the isomorphism
$$
    \wedge^2\mu\colon
    H^2(X,\bC)\longrightarrow\Lambda\otimes\bC.
$$
Let $w\colon\bC^4\to V$ be the map corresponding to $(X, \mu)$, and express it as a matrix
$$
    \Pi = \begin{pmatrix}
        \Pi_{ij}
    \end{pmatrix}\in\uM_{2\times 4}(\bC)
$$
with respect to a given basis of $V$. Write $\mu = (\mu^1, \mu^2, \mu^3, \mu^4)$. Following \cite{Shioda}*{Sections~2 and 3}, the period can be explicitly realized as the $1$-dimensional subspace spanned by
$$
    v_w\colonequals
    \sum_{1\leq i<j\leq4}\det
    \begin{pmatrix}
        \Pi_{1i} & \Pi_{1j} \\
        \Pi_{2i} & \Pi_{2j}
    \end{pmatrix}
    \mu^i\wedge\mu^j
    \in\Lambda\otimes\bC.
$$
By construction, 
$$
    v_w\cdot v_w
    = \det\begin{pmatrix}
        \Pi \\ \Pi
    \end{pmatrix} = 0, \qquad
    v_w\cdot\overline{v}_w
    = \det\begin{pmatrix}
        \Pi \\ \overline{\Pi}
    \end{pmatrix} \neq 0.
$$

\begin{lem}
\label{lem:B_to_period}
The assignment $w\mapsto v_w$ defines an isomorphism
$$
    \cB\cong\left\{
        [v]\in\bP(\Lambda\otimes\bC)
    \;\middle|\;
        v\cdot v=0,\; v\cdot\overline{v}\neq0
    \right\},
$$
under which the complex $2$-tori with admissible marking form the period domain
$$
    \cD\colonequals\left\{
        [v]\in\bP(\Lambda\otimes\bC)
    \;\middle|\;
        v\cdot v=0,\; v\cdot\overline{v}>0
    \right\}\subseteq\cB.
$$
\end{lem}
\begin{proof}
The map $w\mapsto v_w$ coincides with the Pl\"ucker embedding $\Gr(2,4)\hookrightarrow\bP^5$, whose image is defined by the quadric $v\cdot v=0$. In this setting, the non-degeneracy condition $w\wedge\overline{w}\neq0$ translates to $v_w\cdot\overline{v_w}\neq0$. This establishes the identification for $\cB$. The statement regarding admissible markings follows from \cite{Shioda}*{Equation~(2.4)}.
\end{proof}

\subsection{A lattice-theoretic lemma}
\label{subsec:condition-mirror}

If an abelian surface $X$ has a mirror partner $Y$, then by definition its transcendental lattice $T(X)$ is isomorphic to $N(Y)$, and hence contains a copy of the hyperbolic plane $U$. We now establish a lattice-theoretic lemma that describes several conditions equivalent to this property. The coarse moduli spaces will be constructed for lattice-polarizations satisfying these equivalent properties.

\begin{deflem}
\label{deflem:mirror}
A primitive sublattice $M\subseteq U^{\oplus3}$ of signature $(1,r-1)$ is said to satisfy Condition~$\diamondsuit$\footnote{
    The symbol $\diamondsuit$ is intended to suggest a mirror.
} if any of the following equivalent statements holds:
\begin{enumerate}[label=\textup{(\arabic*)}]
\item\label{mirror:U-in-Mperp}
    The orthogonal complement $M^\perp\subseteq U^{\oplus 3}$ contains a copy of $U$.
\item\label{mirror:M-in-2U}
    $M$ admits a primitive embedding into $U^{\oplus2}$.
\item\label{mirror:rank}
    $r = \rank(M)\leq 3$, and if $r = 3$, then $M\cong U\oplus\bZ(-2n)$ for some positive integer~$n$.
\item\label{mirror:N-in-Mperp}
    There exists a primitive sublattice $N\subseteq U^{\oplus3}$ such that $M^\perp\cong N\oplus U$.
\item\label{mirror:M-in-Nperp}
    There exists a primitive sublattice $N\subseteq U^{\oplus3}$ such that $N^\perp\cong M\oplus U$.
\item\label{mirror:pair}
    There exists a primitive sublattice $N\subseteq U^{\oplus3}$ with equalities
    $$
        M^\perp = N\oplus U
        \qquad\text{and}\qquad
        N^\perp = M\oplus U.
    $$
\end{enumerate}
\end{deflem}

\begin{proof}
A key fact to be used in the proof is that all embeddings $U\hookrightarrow U^{\oplus3}$ are equivalent under the action of the orthogonal group $\uO(U^{\oplus3})$ \cite{Nik79}*{Theorem~1.14.4}. We first show that \ref{mirror:U-in-Mperp}, \ref{mirror:M-in-2U}, and \ref{mirror:rank} are equivalent.
\begin{itemize}
\item[] \ref{mirror:U-in-Mperp} $\Longleftrightarrow$ \ref{mirror:M-in-2U}:
    If $M^\perp$ contains a copy of $U$, then, under the action of $\uO(U^{\oplus3})$, we may assume that this copy coincides with the last summand of $U^{\oplus3}$. It follows that $M$ is a primitive sublattice of $U^\perp = U^{\oplus 2}$. Conversely, if $M\subseteq U^{\oplus 2}$, then $M^\perp$ contains the orthogonal complement of $U^{\oplus 2}$, which is isometric to $U$.
\item[] \ref{mirror:M-in-2U} $\Longleftrightarrow$ \ref{mirror:rank}:
    Suppose $M$ is a primitive sublattice in $U^{\oplus 2}$. Since $M$ and $U^{\oplus 2}$ have signatures $(1, r-1)$ and $(2, 2)$, respectively, we have $r\leq 3$. If $r = 3$, then $M$ has signature $(1, 2)$. This implies that $M^\perp\subseteq U^{\oplus 2}$ is isomorphic to $\bZ(2n)$ for some positive integer~$n$. Identifying $\bZ(2n)$ as a sublattice of one summand of $U^{\oplus 2}$ via $\uO(U^{\oplus 2})$, we conclude that
    $$
        M = \bZ(2n)^\perp\cong U\oplus\bZ(-2n).
    $$
    Conversely, if $M$ has rank $r\leq2$, it admits a primitive embedding into $U^{\oplus2}$; see, for example, \cite{Huy16}*{Proposition~14.1.8}. If $M\cong U\oplus\bZ(-2n)$, the desired embedding can be obtained by extending a primitive embedding $\bZ(-2n)\hookrightarrow U$.
\end{itemize}

Next, we show that
(\ref{mirror:N-in-Mperp} or \ref{mirror:M-in-Nperp})
    $\Longrightarrow$
(\ref{mirror:U-in-Mperp} or \ref{mirror:M-in-2U})
    $\Longrightarrow$
\ref{mirror:pair}
    $\Longrightarrow$
(\ref{mirror:N-in-Mperp} and \ref{mirror:M-in-Nperp}).
Note that the implication
\ref{mirror:N-in-Mperp}
    $\Longrightarrow$
\ref{mirror:U-in-Mperp}
and the one
\ref{mirror:pair}
    $\Longrightarrow$
(\ref{mirror:N-in-Mperp} and \ref{mirror:M-in-Nperp})
are immediate.
\begin{itemize}
\item[] \ref{mirror:M-in-Nperp}$\implies$\ref{mirror:M-in-2U}:
By applying $\uO(U^{\oplus3})$, we may identify the $U$-summand of $M\oplus U$ with one of the summands of $U^{\oplus 3}$, thereby realizing $M$ as a sublattice of $U^\perp = U^{\oplus 2}$. This embedding is primitive since $N^\perp\cong M\oplus U$ is primitive.

\item[] \ref{mirror:M-in-2U}$\implies$\ref{mirror:pair}:
Let $N\subseteq U^{\oplus2}$ be the orthogonal complement of $M\subseteq U^{\oplus2}$. This gives equalities $M^\perp = N$ and $N^\perp = M$ as sublattices in $U^{\oplus 2}$. Adding one more copy of $U$, we obtain $M^\perp = N\oplus U$ and $N^\perp = M\oplus U$ as sublattices of $U^{\oplus 3}$.
\end{itemize}
This completes the proof.
\end{proof}

\begin{cor}
\label{cor:numGroth-trans}
Let $X$ and $Y$ be abelian surfaces. Then $T(X)$ is isomorphic to $N(Y)$ if and only if $N(X)$ is isomorphic  to $T(Y)$.
\end{cor}
\begin{proof}
Choose an isometry
$
    \bigoplus_{i = 0}^2H^{2i}(X, \bZ)
    \cong U^{\oplus 4}
$
under which $H^2(X,\bZ)\cong U^{\oplus 3}$ coincides with the first three copies of $U$. With this identification, we have
$$
    N(X) = \NS(X)\oplus U, \qquad
    T(X) = \NS(X)^{\perp U^{\oplus 3}}.
$$
Make the same choice for $Y$ so that similar identifications hold for it. If $T(X)\cong N(Y)$, then
$$
    \NS(X)^{\perp U^{\oplus 3}}
    = T(X)
    \cong N(Y)
    = \NS(Y)\oplus U.
$$
This corresponds to \ref{mirror:N-in-Mperp} in Definition-Lemma~\ref{deflem:mirror} with $M = \NS(X)$ and $N = \NS(Y)$. Condition~\ref{mirror:M-in-Nperp} of the same lemma then yields
$
    \NS(Y)^{\perp U^{\oplus 3}}
    \cong\NS(X)\oplus U.
$
It follows that
$$
    N(X) = \NS(X)\oplus U
    \cong \NS(Y)^{\perp U^{\oplus 3}}
    = T(Y)
$$
establishing $N(X)\cong T(Y)$. The proof for the converse is similar.
\end{proof}

\subsection{Lattice-polarization and complex moduli}
\label{subsec:lattice-pol-AbS}

We now introduce lattice-polarized abelian surfaces, in analogy with Nikulin's lattice-polarized K3~surfaces \cite{Nik79_K3}; see also \cite{Dol96}. The key difference is that the marking is taken on the first cohomology group rather than the second, since the period of an abelian surface does not uniquely determine the surface itself.

Let $M\subseteq\Lambda$  be a primitive sublattice of signature $(1, r-1)$. We say that a marked abelian surface $(X, \mu)$ is \emph{$M$-polarized} if
\begin{itemize}
    \item the marking $\mu\colon H^1(X, \bZ)\rightarrow L$ is admissible, and
    \item under the isometry 
    $
        \wedge^2\mu\colon
        H^2(X,\bZ)\rightarrow\Lambda,
    $
    we have $(\wedge^2\mu)^{-1}(M)\subseteq\NS(X)$.
\end{itemize}
Fix a vector $h\in M\otimes\bR$ with $h^2 > 0$. We say that $(X, \mu)$ is \emph{$(M, h)$-polarized} if, in addition,
\begin{itemize}
    \item the class $(\wedge^2\mu)^{-1}(h)\in\NS(X)$ is ample.
\end{itemize}
For each pair $(M, h)$ as above, we define a functor
$$
    \cM_{M,\,h}\colon
    \left(\text{Complex analytic varieties}\right)^{\rm op}
    \longrightarrow
    \mathbf{Set}
$$
which maps a space $S$ to the set of isomorphism classes of marked families of abelian surfaces $(\pi\colon\cX\rightarrow S,\;\mu)$ such that for each $s\in S$, the fiber $(\cX_s, \mu_s)$ is an $(M, h)$-polarized abelian surface. Here, two such families
$
    \left(
        \pi\colon\cX\to S,\; \mu
    \right)
$
and
$
    \left(
        \pi'\colon\cX'\to S,\; \mu'
    \right)
$
are said to be \emph{isomorphic} if there exists an $S$-isomorphism $f\colon\cX\to\cX'$ such that for any $s\in S$, the following diagram commutes:
$$\xymatrix@R=2em@C=8em{
    & \NS(\cX_s)\\
    M\ar[ur]^-{
        \left(
            \wedge^2\mu_s
        \right)^{-1}\big|_M
    }\ar[dr]_-{
        \left(
            \wedge^2\mu'_s
        \right)^{-1}\big|_M
    } & \\
    & \NS(\cX'_s)\ar[uu]_-{
        f_s^*\big|_{\NS(\cX'_s)}
    }
}$$

We now proceed to construct the coarse moduli space for the functor $\cM_{M,\, h}$ under the assumption that the lattice $M\subseteq\Lambda$ satisfies Condition~\hyperref[deflem:mirror]{$\diamondsuit$}.

Let $M^\perp\subseteq\Lambda$ denote the orthogonal complement of $M$, and consider the period domain
$$
    \cD_M\colonequals\left\{
        [v]\in\bP\left(M^\perp\otimes\bC\right)
    \;\middle|\;
        v\cdot v = 0,\; v\cdot\overline{v} > 0
    \right\}.
$$
This domain has two connected component, distinguished by the orientation of the positive-definite plane $\left<\mathrm{Re}(v), \mathrm{Im}(v)\right>\subseteq M^\perp\otimes\bR$. It carries a natural action of the discrete group
$$
    \Gamma_M\colonequals\{
        \phi\in\SO(\Lambda)
    \mid
        \phi|_M = \mathrm{id}_M
    \}.
$$
Let $\SO^+(\Lambda)\subseteq\SO(\Lambda)$ be the index-two subgroup preserving the orientation of positive-definite $3$-planes in $\Lambda\otimes\bR$. Then the subgroup
$$
    \Gamma_M^+\colonequals\Gamma_M\cap\SO^+(\Lambda)
$$
preserves each connected component of $\cD_M$.

In what follows, for a marked complex torus $(X,\mu)$, we denote by
$$
    p(X,\mu)\in\cB\subseteq\bP(\Lambda\otimes\bC)
$$
its period, namely the point representing the image of the line $H^{2,0}(X)\subseteq H^2(X,\bC)$ under the isometry
$$
    \wedge^2\mu\colon H^2(X,\bC)
    \longrightarrow\Lambda\otimes\bC.
$$
Recall that, by Lemma~\ref{lem:B_to_period}, every $p\in\cB$ is the period of some marked complex torus $(X, \mu)$, and there can be written as $p = p(X,\mu)$.

\begin{lem}
\label{lem:D_M}
Let $M\subseteq\Lambda$ be a primitive sublattice of signature $(1,r-1)$. Then
$$
    \cD_M = \{\,
        p(X,\mu)\in\cD
    \mid
        (X,\mu)\text{ is }M\text{-polarized}
    \,\}
$$
\end{lem}
\begin{proof}
Let $(X,\mu)$ be an $M$-polarized abelian surface. By definition, $\mu$ is admissible and satisfies $(\wedge^2\mu)^{-1}(M)\subseteq\NS(X)$. The former condition implies $p(X,\mu)\in\cD$ by Lemma~\ref{lem:B_to_period}, and the latter ensures $p(X,\mu)\in\cD_M$. Hence,
$$
    \cD_M \supseteq \{\,
        p(X,\mu)\in\cD
    \mid
        (X,\,\mu)\text{ is }M\text{-polarized}
    \,\}.
$$
Conversely, every $p\in\cD_M\subseteq\cB$ is the period of some marked complex torus $(X,\mu)$. The condition $p\in\cD_M$ implies
$$
    H^{2,0}(X)\subseteq (\wedge^2\mu)^{-1}(M)^{\perp},
    \quad\text{or equivalently,}\quad
    \NS(X)\supseteq (\wedge^2\mu)^{-1}(M).
$$
Thus $(X,\mu)$ is an $M$-polarized abelian surface. This proves the reverse inclusion.
\end{proof}

\begin{lem}
\label{lem:D_M-plus-minus}
Let $M\subseteq\Lambda$ be a primitive sublattice of signature $(1,r-1)$ satisfying Condition~\hyperref[deflem:mirror]{$\diamondsuit$}, and fix a vector $h\in M\otimes\bR$ with $h^2>0$. Then the two connected components of $\cD_M$ correspond precisely to the subsets
\begin{align*}
    \cD_M^+ &\colonequals\{
        p(X,\mu)\in\cD_M
    \mid
        (X,\mu)\text{ is }(M,h)\text{-polarized}
    \}, \\
    \cD_M^- &\colonequals\left\{
        p(X,\mu)\in\cD_M
    \;\middle|\;
        (X,\mu)\text{ is }(M,-h)\text{-polarized}
    \right\}.
\end{align*}
\end{lem}
\begin{proof}
Lemma~\ref{lem:D_M} asserts that $\cD_M$ consists precisely of the periods of $M$-polarized abelian surfaces, so the two subsets can be identified respectively as
\begin{align*}
    \cD_M^+ &= \left\{
        p(X,\mu)\in\cD_M
    \;\middle|\;
        (\wedge^2\mu)^{-1}(h)\in\NS(X)\text{ is ample}
    \right\}, \\
    \cD_M^- &= \left\{
        p(X,\mu)\in\cD_M
    \;\middle|\;
        (\wedge^2\mu)^{-1}(h)\in\NS(X)\text{ is anti-ample}
    \right\}.
\end{align*}
For every $p(X,\mu)\in\cD_M$, the class
$
    (\wedge^2\mu)^{-1}(h)\in\NS(X)
$
is positive, and hence is either ample or anti-ample. Thus the subsets $\cD_M^+$ and $\cD_M^-$ are disjoint, and their union is all of $\cD_M$. Both subsets are open, since (anti-)ampleness is an open condition in families. Consequently, if each subset is non-empty, they are exactly the two connected components of $\cD_M$. Hence, to complete the proof, it remains to construct $M$-polarized abelian surfaces $(X, \mu)$ and $(X', \mu')$ such that $(\wedge^2\mu)^{-1}(h)$ is ample and $(\wedge^2\mu')^{-1}(h)$ is anti-ample.

Let $(X, \mu\colon H^1(X,\bZ)\to L)$ be an $M$-polarized abelian surface, and first consider the case where $(\wedge^2\mu)^{-1}(h)$ is ample. Since $M$ satisfies Condition~\hyperref[deflem:mirror]{$\diamondsuit$}, it admits a primitive embedding into $U^{\oplus 2}$, which in turn embeds into $U^{\oplus 3}$ uniquely up to the action of $\uO(U^{\oplus 3})$. Consider the isometry acting as $-\mathrm{id}$ on the $U^{\oplus 2}$ containing $M$ and trivially on the remaining copy of~$U$. This isometry has determinant~$1$ and preserves the orientation of positive $3$-planes in $\Lambda$. By \cite{Shioda}*{Lemma 1}, it is therefore induced by an admissible isomorphism of $L$. Composing~$\mu$ with this isomorphism yields a new marking $\mu'$ under which $M$ is still mapped into $\NS(X)$, but~$h$ is sent to an anti-ample class. The same argument applies when $(\wedge^2\mu)^{-1}(h)$ is anti-ample. This completes the proof.
\end{proof}

\begin{lem}
\label{lem:bij_quotient}
Retain the hypothesis of Lemma~\ref{lem:D_M-plus-minus}, and let $\cD_M^+\subseteq\cD_M$ be the connected component specified there. Then there is a bijection
$$
    \cM_{M,\,h}(\mathrm{Spec}(\bC))
    \cong
    \cD_M^+/\Gamma_M^+.
$$
\end{lem}
\begin{proof}
If two $(M, h)$-polarized abelian surfaces $(X_1,\mu_1)$ and $(X_2,\mu_2)$ are isomorphic, then by definition there exists an isomorphism $f\colon X_1\rightarrow X_2$ such that
\begin{equation}
\label{eqn:(M,h)-pol_isom}
    (\wedge^2\mu_1)^{-1}\big|_M
    = f^*\big|_{\NS(X_2)}\circ(\wedge^2\mu_2)^{-1}\big|_M.
\end{equation}
The induced isomorphism
$$
    f^*\colon H^2(X_2,\bC)\longrightarrow H^2(X_1,\bC)
$$
preserves the Hodge structure and therefore maps $H^{2,0}(X_2)$ to $H^{2,0}(X_1)$. In particular, the periods of the two surfaces satisfy
\begin{align*}
    p(X_1,\mu_1)
    &= \wedge^2\mu_1(H^{2,0}(X_1)) \\
    &= \wedge^2\mu_1\circ f^*(H^{2,0}(X_2)) \\
    &= \wedge^2\mu_1\circ f^*
        \circ (\wedge^2\mu_2)^{-1}
        \circ \wedge^2\mu_2(H^{2,0}(X_2)) \\
    &= \wedge^2\mu_1\circ f^*
        \circ (\wedge^2\mu_2)^{-1}(p(X_2,\mu_2)).
\end{align*}
Let us abbreviate this by writing
$$
    p(X_1,\mu_1) = \phi(p(X_2,\mu_2)), \qquad
    \phi = \wedge^2\mu_1\circ f^*
        \circ (\wedge^2\mu_2)^{-1}\in\SO(\Lambda).
$$

Relation~\eqref{eqn:(M,h)-pol_isom} implies that $\phi$ acts trivially on $M$, and hence $\phi\in\Gamma_M$. In particular, $\phi$ fixes one positive direction since $M$ has signature $(1, r-1)$. Moreover, $\phi$ carries the oriented positive plane spanned by the real and imaginary parts of $p(X_2,\mu_2)$ to the corresponding plane given by $p(X_1,\mu_1)$. Thus $\phi$ preserves the orientation of positive $3$-planes in $\Lambda\otimes\bR$, and therefore $\phi\in\Gamma_M^+$. Consequently, sending an abelian surface to its period induces a well-defined map
$$
    \cM_{M,\,h}(\mathrm{Spec}(\bC))
    \longrightarrow
    \cD_M^+/\Gamma_M^+.
$$
This map is surjective, since every point on $\cD_M^+$ arises as the period of an $(M, h)$-polarized abelian surface. It remains to show that the map is injective.

Assume that two $(M, h)$-polarized abelian surfaces $(X_1,\mu_1)$ and $(X_2,\mu_2)$ are mapped to the same point in $\cD_M^+/\Gamma_M^+$. Then
$
    p(X_1,\mu_1) = \phi(p(X_2,\mu_2))
$
for some $\phi\in\Gamma_M^+$. Consider the isometry
$$\xymatrix{
    g\colon H^2(X_2,\bZ)\ar[rr]^-{\wedge^2\mu_2}
    && \Lambda\ar[r]^-\phi
    & \Lambda\ar[rr]^-{(\wedge^2\mu_1)^{-1}}
    && H^2(X_1,\bZ).
}$$
The relation
$
    p(X_1,\mu_1) = \phi(p(X_2,\mu_2))
$
implies that $g$ preserves the Hodge structure. Moreover, since $\phi\in\SO(\Lambda)$, we have $\det(g) = 1$ in the sense of \cite{Shioda}*{Equation~(1.9)}. By \cite{Shioda}*{\S4, Theorem~1}, after replacing $\phi$ by $-\phi$ if necessary, the isometry~$g$ is induced by an isomorphism $f\colon X_1\rightarrow X_2$. Because $\phi$ acts trivially on $M$, the isomorphism $f$ satisfies relation~\eqref{eqn:(M,h)-pol_isom}. Hence $(X_1,\mu_1)$ and $(X_2,\mu_2)$ are isomorphic. This proves injectivity, which completes the proof.
\end{proof}

\begin{defn}
\label{defn:cpx-moduli}
Let $\cD_M^+\subseteq\cD_M$ be the connected component specified in Lemma~\ref{lem:D_M-plus-minus} for a primitive sublattice $M\subseteq\Lambda$ of signature $(1, r-1)$ satisfying Condition~\hyperref[deflem:mirror]{$\diamondsuit$} and a vector $h\in M\otimes\bR$ with $h^2>0$. We define
$$
    \cM_\cpx(M)\colonequals\cD_M^+/\Gamma_M^+.
$$
For an abelian surface $X$, we further define
$$
    \cM_\cpx(X)\colonequals\cM_\cpx(\NS(X)).
$$
\end{defn}

Note that $\cM_\cpx(M)$ is a quasi-projective variety by Baily--Borel \cite{BB66}, whose dimension is equal to $4 - \rank(M)$. In particular, if $X$ is an abelian surface with Picard number $\rho(X)$, then $\cM_\cpx(X)$ is a quasi-projective variety of dimension $4 - \rho(X)$.

\begin{thm}
\label{thm:cplx_moduli}
The quotient
$
    \cM_\cpx(M) = \cD_M^+/\Gamma_M^+
$
defined above is a coarse moduli space for the moduli functor $\cM_{M,\,h}$.
\end{thm}

\begin{proof}
We begin by constructing a natural transformation of functors
$$
    \eta\colon\cM_{M,\,h}(-)
    \longrightarrow
    \Hom(-,\,\cD_M^+/\Gamma_M^+).
$$
Let $(\pi\colon\cX\rightarrow S)\in\cM_{M,\,h}(S)$ be a family of $(M, h)$-polarized abelian surfaces. Since $\cB$ is a fine moduli space of marked complex tori, this family is the pullback of the universal family over $\cB$ along a unique morphism $f\colon S\rightarrow\cB$. Moreover, since all fibers are $(M, h)$-polarized, the image of $f$ lies in $\cD_M^+\subseteq\cB$. Thus for each base space $S$, we obtain a map
$$
    \eta_S\colon\cM_{M,\,h}(S)
    \longrightarrow
    \Hom(S,\,\cD_M^+/\Gamma_M^+)
$$
sending a family $(\pi\colon\cX\rightarrow S)$ to the composition
$$\xymatrix{
    S\ar[r]^-f
    & \cD_M^+\ar[r]
    & \cD_M^+/\Gamma_M^+.
}$$
This construction is functorial in $S$, providing the desired natural transformation.

To prove that $\cD_M^+/\Gamma_M^+$ is a coarse moduli space, we need to verify that the natural transformation $\eta$ satisfies the following properties:
\begin{enumerate}[label=(\roman*)]
    \item The following map is bijective:
    $$
        \eta_{\Spec(\bC)}\colon
        \cM_{M,\,h}(\Spec(\bC))
        \longrightarrow
        \Hom(\Spec(\bC),\,\cD_M^+/\Gamma_M^+)
    $$

    \item For every natural transformation
    $
        \tau\colon\cM_{M,\,h}(-)
        \longrightarrow
        \Hom(-,\,N),
    $
    there exists a unique morphism
    $
        \psi\colon\cD_M^+/\Gamma_M^+
        \longrightarrow N
    $
    such that the following diagram commutes:
    \begin{equation}
    \label{eqn:univ_prop}
    \vcenter{\xymatrix{
        \cM_{M,\,h}(-)\ar[r]^-\eta\ar[dr]_-\tau
        & \Hom(-,\,\cD_M^+/\Gamma_M^+)\ar[d]^-{\psi\,\circ\,-} \\
        & \Hom(-,\,N)
    }}
    \end{equation}
\end{enumerate}
The first property is established in Lemma~\ref{lem:bij_quotient}, so it remains to verify the second.

More specifically, we need to construct the morphism $\psi$. To this end, consider the family of $(M, h)$-polarized abelian surfaces
$
    \pi_M\colon\mathcal{U}_M^+
    \rightarrow\cD_M^+
$
obtained by the pullback of the universal family over $\cB$. Applying $\tau$ to $\pi_M$ yields a morphism
$
    \psi'\colon\cD_M^+
    \rightarrow N.
$
Every isometry $\phi\in\Gamma_M^+$ corresponds to an automorphism $\phi\colon\cD_M^+\rightarrow\cD_M^+$, giving rise to the commutative diagram
$$\xymatrix@C=5em{
    \cM_{M,\,h}(\cD_M^+)
        \ar[d]_-{\phi^*}
        \ar[r]^-{\tau_{\cD_M^+}}
    & \Hom(\cD_M^+, N)
        \ar[d]^-{\phi\,\circ\,-} \\
    \cM_{M,\,h}(\cD_M^+)
        \ar[r]^-{\tau_{\cD_M^+}}
    & \Hom(\cD_M^+, N).
}$$
Since $\phi^*$ leaves the isomorphism class of $\pi_M$ invariant, it follows that $\phi\circ\psi' = \psi'$ from the diagram. Hence, $\psi'$ is invariant under the action of $\Gamma_M^+$, and consequently descends to a morphism
$
    \psi\colon\cD_M^+/\Gamma_M^+
    \rightarrow N.
$
It is straightforward to verify that diagram~\eqref{eqn:univ_prop} commutes and that $\psi$ is uniquely determined. The details are left to the reader.
\end{proof}

\begin{rmk}
As pointed out by Alexeev and Engel \cite{AE25}, \cite{Dol96}*{Theorem~3.1} is incorrect under Dolgachev’s original definition of lattice-polarized K3 surfaces, which was corrected in \cite{AE25}. This issue does not arise for abelian surfaces, since they contain no $(-2)$-curves, and hence their ample cones coincide with positive cones.
\end{rmk}

\section{Mirror pairs of abelian surfaces}
\label{sec:mirror-pair}

Mirror symmetry naturally identifies the complex moduli space of an abelian surface with the stringy K\"ahler moduli space of its mirror partner. In this section, we construct the stringy K\"ahler moduli space via Bridgeland stability conditions, paralleling Bayer and Bridgeland's construction for K3~surfaces \cite{BB17}*{Section~7}, and then establish this identification. We next introduce an involution on the stringy K\"ahler moduli space which, assuming the existence of a mirror partner, recovers the involution on the complex moduli space sending an abelian surface to its dual. Finally, we develop criteria for existence of mirror partners and for determining when two abelian surfaces form a mirror pair.

\subsection{Stringy K\"ahler moduli spaces}
\label{subsec:stringy-kah-moduli_AbS}

We begin by recalling the necessary background on Bridgeland stability conditions for abelian surfaces. For general definitions and further details, the reader is referred to \cite{Bri08}*{Section~15}, or more recent reference such as \cite{FLZ22}*{Section~2}.

Consider an abelian surface $X$ and denote by $\Db(X)$ the bounded derived category of coherent sheaves on it. A \emph{numerical stability condition} on $\Db(X)$ is a pair $\sigma = (Z,\sP)$ which consists of the following data.
\begin{itemize}
\item[] \textbf{Central charge:} A group homomorphism
$$
    Z\colon K_0(X)\longrightarrow\bC,
$$
with $K_0(X)$ denoting the Grothendieck group of $\Db(X)$, factoring through the Mukai vector $v\colon K_0(X)\longrightarrow N(X)$. It can be expressed in the form
$
    Z(-) = (\gamma_Z, v(-))
$
for some $\gamma_Z\in N(X)\otimes\bC$.
\item[] \textbf{Slicing:} A collection of full additive subcategories $\sP(\theta)\subseteq\Db(X)$, one for each real number $\theta$, subject to the following conditions:
\begin{itemize}
	\item[$\bullet$] $\sP(\theta+1) = \sP(\theta)[1]$,
	\item[$\bullet$] $\Hom(A_1, A_2) = 0$ for all $A_i\in \sP(\theta_i)$, $i=1,2$, with $\theta_1>\theta_2$, and
	\item[$\bullet$] every nonzero object $E\in\Db(X)$ can be filtered into exact triangles
    $$\xymatrix@C=.5em{
        & 0 \ar[rrrr] &&&& E_1 \ar[rrrr] \ar[dll] &&&& E_2
        \ar[rr] \ar[dll] && \cdots \ar[rr] && E_{n-1}
        \ar[rrrr] &&&& E \ar[dll]  &  \\
        &&& A_1 \ar@{-->}[ull] &&&& A_2 \ar@{-->}[ull] &&&&&&&& A_n \ar@{-->}[ull] 
    }$$
	where $A_i\in \sP(\theta_i)$ and $\theta_1>\theta_2>\cdots>\theta_n$.
\end{itemize}
An object of $\sP(\theta)$ is said to be \emph{semistable of phase $\theta$.}
\end{itemize}
The central charge and slicing are required to be compatible in the sense that, for every nonzero object $E\in\sP(\theta)$, its central charge satisfies $Z(E)\in\bR_{>0}\,e^{i\pi\theta}$. In addition, the central charge needs to satisfy the following property.
\begin{itemize}
    \item[] \textbf{Support property:} There exists a quadratic form $Q$ on $N(X)\otimes\bR$ that is negative-definite on the subspace where $Z$ vanishes and $Q(v(E))\geq 0$ for every semistable $E$.
\end{itemize}

The set of numerical stability conditions $\Stab(X)$ carries the structure of a complex manifold, which a priori may have several components. For abelian surfaces, since there is no spherical object, the space $\Stab(X)$ is connected and simply connected due to \cite{HMS08}*{Theorem~3.15}. Within this space, the \emph{reduced} stability conditions, namely those whose central charges $Z(-) = (\gamma_Z, v(-))$ satisfy $\gamma_Z\cdot\gamma_Z = 0$, form a submanifold
$$
    \Stab_\red(X)
    \subseteq\Stab(X).
$$
The space $\Stab(X)$ carries a right action of $\smash{\widetilde{\GL}_2^+}(\bR)$, the universal cover of $\GL_2^+(\bR)$, and a left action of $\Aut\Db(X)$, with the two actions commuting. Under these actions, $\Stab_\red(X)$ is preserved by the subgroup
$
    \bC\subseteq
    \smash{\widetilde{\GL}_2^+}(\bR)
$
and by the entire $\Aut\Db(X)$. We denote by
$$
    \Aut_0\Db(X)\subseteq\Aut\Db(X).
$$
the subgroup of autoequivalences acting trivially on $\Stab(X)$. It is generated by tensoring with degree-zero line bundles and by pullbacks of automorphisms of $X$ acting trivially on the total cohomology $H^*(X,\bZ)$.

An autoequivalence is called \emph{Calabi--Yau} if the induced Hodge isometry on $H^{\rm ev}(X, \bZ)$ acts trivially on the transcendental lattice $T(X)$, or equivalently, if it respects the Serre duality pairing \cite{BB17}*{Definition~7.1 and Appendix}. Such autoequivalences form a subgroup
$$
    \Aut_\CY\Db(X)\subseteq\Aut\Db(X).
$$
Note that $\Aut_0\Db(X)\subseteq\Aut_\CY\Db(X)$. We denote the corresponding quotient by
$$
    \overline{\Aut}_\CY\Db(X)
    \colonequals
    \Aut_\CY\Db(X) / \Aut_0\Db(X).
$$

\begin{defn}
\label{defn:stringy_Kah_mod}
The stringy K\"ahler moduli space of an abelian surface $X$ is defined as
$$
    \cM_\Kah(X)\colonequals\left(
        \Stab_\red(X)/\bC
    \right) \big/ \left(\,
        \overline{\Aut}_\CY\Db(X)/\bZ[2]
    \,\right).
$$
\end{defn}

The space $\cM_\Kah(X)$ admits a more concrete description. To a stability condition with central charge $Z(-) = (\gamma_Z, v(-))$, one can associate a class
$
    [\gamma_Z]\in\bP(N(X)\otimes\bC).
$
As a consequence of \cite{Bri08}*{Theorem~15.2}, this assignment induces an isomorphism
\begin{equation}
\label{eqn:covering}
\xymatrix{
    \Stab_\red(X)/\bC
    \ar[r]^-\sim &
    \cQ^+(X)
    \subseteq\bP(N(X)\otimes\bC).
}\end{equation}
Here, $\cQ^+(X)$ denotes the connected component of the period domain
$$
    \cQ(X)\colonequals\{
        [\gamma]\in\bP(N(X)\otimes\bC)
    \mid
        \gamma\cdot\gamma = 0,\;
        \gamma\cdot\overline{\gamma} > 0
    \}
$$
containing the class $\left[e^{ih} = 1 + ih - \frac{1}{2}h^2\right]$ for some ample $h\in\NS(X)\otimes\bR$. This component is preserved by the action of $\SO^+(N(X))^*$, the special orthogonal group of $N(X)$ preserving the orientation of positive $2$-planes in $N(X)\otimes\bR$ and acting trivially on the discriminant group $N(X)^*/N(X)$.

\begin{prop}
\label{prop:str-kah-moduli}
For an abelian surface $X$, there is a natural isomorphism
$$
    \cM_\Kah(X)\cong
    \cQ^+(X) \,/ \,\SO^+(N(X))^*.
$$
\end{prop}
\begin{proof}
There is an isomorphism $\Stab_\red(X)/\bC\cong\cQ^+(X)$ from \eqref{eqn:covering}. It remains to show that
\begin{equation}
\label{eqn:Aut_CY-to-SO+}
     \overline{\Aut}_\CY\Db(X)/\bZ[2]
     \cong
     \SO^+(N(X))^*.
\end{equation}
Consider the homomorphism
$$
    \Aut\Db(X)
    \longrightarrow
    \uO_{\rm Hdg}(H^{\rm ev}(X,\bZ))
$$
induced by the actions of autoequivalences on the even cohomology. Its kernel is isomorphic to $\Aut_0\Db(X)\times\bZ[2]$. By \cite{Yos09}*{Proposition~4.5}, its image is
$
    \SO_{\rm Hdg}^+(H^{\rm ev}(X,\bZ)),
$
the group of Hodge isometries of $H^{\rm ev}(X,\bZ)$ with determinant~$1$ that preserve the orientation of positive $4$-planes in $H^{\rm ev}(X,\bR)$.

Restricting the above homomorphism to $\Aut_\CY\Db(X)$ yields the identification
$$
    \mathrm{Im}\left(
        \Aut_\CY\Db(X)
        \rightarrow
        \uO_{\rm Hdg}(H^{\rm ev}(X,\bZ))
    \right) = \left\{
        \phi\in\SO_{\rm Hdg}^+(H^{\rm ev}(X,\bZ))
    \mid
        \phi|_{T(X)} = \mathrm{id}_{T(X)}
    \right\}.
$$
The left-hand side equals $\overline{\Aut}_\CY\Db(X)/\bZ[2]$. Moreover, restricting a Hodge isometry of the even cohomology to $N(X)$ defines an injection
$$
    \left\{
        \phi\in\SO_{\rm Hdg}^+(H^{\rm ev}(X,\bZ))
    \mid
        \phi|_{T(X)} = \mathrm{id}_{T(X)}
    \right\}
    \lhook\joinrel\longrightarrow
    \SO^+(N(X))^*.
$$
This map is surjective, and hence an isomorphism, since every element of $\SO^+(N(X))^*$ can be extended to $H^{\rm ev}(X,\bZ)$ acting trivially on $T(X)$; see, for example, \cite{Huy16}*{Proposition~14.2.6}. This establishes the desired isomorphism~\eqref{eqn:Aut_CY-to-SO+}.
\end{proof}

Note that for an abelian surface $X$ with Picard number $\rho(X)$, the space $\cM_\Kah(X)$ is a quasi-projective variety of dimension $\rho(X)$ by Proposition~\ref{prop:str-kah-moduli} and Baily--Borel \cite{BB66}.

\begin{rmk}
\label{rmk:orbifold}
As complex varieties, one can replace $\overline{\Aut}_\CY\Db(X)$ with $\Aut_\CY\Db(X)$ in Definition~\ref{defn:stringy_Kah_mod} and write
$$
    \cM_\Kah(X) = \left(
        \Stab_\red(X)/\bC
    \right) \big/ \left(\,
        \Aut_\CY\Db(X)/\bZ[2]
    \,\right),
$$
making it compatible with Bayer and Bridgeland's corresponding definition for K3~surfaces \cite{BB17}*{Section~7}. However, this formulation is not suitable for abelian surfaces if one wants to regard $\cM_\Kah(X)$ as an orbifold, since the subgroup $\Aut_0\Db(X)\subseteq\Aut_\CY\Db(X)$ that fixes all stability conditions is non-discrete.
\end{rmk}

\subsection{Mirror pairs and symplectic duals}
\label{subsec:mirror-symp-dual}

We now show that, if an abelian surface $X$ admits a mirror partner $Y$, then there exist natural isomorphisms
$$
    \cM_\cpx(X)\cong\cM_\Kah(Y)
    \qquad\text{and}\qquad
    \cM_\Kah(X)\cong\cM_\cpx(Y).
$$

\begin{proof}[Proof of Theorem~\ref{thm:mirror-moduli}]
Recall from Definition~\ref{defn:cpx-moduli} and the beginning of Section~\ref{subsec:lattice-pol-AbS} that
$$
    \cM_\cpx(X) = \cD_{\NS(X)}^+/\Gamma_{\NS(X)}^+
$$
where $\cD_{\NS(X)}^+$ is one of the connected components of
\begin{align*}
    \cD_{\NS(X)} &= \left\{
        [v]\in\bP\left(\NS(X)^{\perp\Lambda}\otimes\bC\right)
    \;\middle|\;
        v\cdot v = 0,\; v\cdot\overline{v} > 0
    \right\} \\
    &\cong \left\{
        [v]\in\bP\left(T(X)\otimes\bC\right)
    \;\middle|\;
        v\cdot v = 0,\; v\cdot\overline{v} > 0
    \right\}
\end{align*}
and
$$
    \Gamma_{\NS(X)}^+ = \{
        \phi\in\SO^+(\Lambda)
    \mid
        \phi|_{\NS(X)} = \mathrm{id}_{\NS(X)}
    \}.
$$
On the other hand, Proposition~\ref{prop:str-kah-moduli} shows that
$$
    \cM_\Kah(Y)\cong
    \cQ^+(Y) \,/ \,\SO^+(N(Y))^*
$$
where $\cQ^+(Y)$ is one of the connected components of
$$
    \cQ(Y) = \{
        \gamma\in\bP(N(Y)\otimes\bC)
    \mid
        \gamma\cdot\gamma = 0,\;
        \gamma\cdot\overline{\gamma} > 0
    \}.
$$
and $\SO^+(N(Y))^*$ is the group of isometries of $N(Y)$ with determinant~$1$ that preserve the orientation of positive $2$-planes in $N(Y)\otimes\bR$ and act trivially on $N(Y)^*/N(Y)$.

Since $Y$ is a mirror partner of $X$, there is an isometry $T(X)\cong N(Y)$, which in turn induces an isomorphism
\begin{equation}
\label{eqn:isom-period-domain}
\xymatrix{
    \cD_{\NS(X)}^+
    \ar[r]^-\sim &
    \cQ^+(Y).
}
\end{equation}
Moreover, every element of $\Gamma_{\NS(X)}^+$ acts on $\NS(X)^{\perp\Lambda}\cong T(X)$, which defines an injection
$$
    \Gamma_{\NS(X)}^+
    \lhook\joinrel\longrightarrow
    \SO^+(T(X))^*.
$$
This map is surjective, and hence an isomorphism, since every element of $\SO^+(T(X))^*$ can be extended to $\Lambda$ acting trivially on $\NS(X)$ \cite{Huy16}*{Proposition~14.2.6}. Combining this with the isometry $T(X)\cong N(Y)$ yields an isomorphism
$$\xymatrix{
    \Gamma_{\NS(X)}^+
    \ar[r]^-\sim &
    \SO^+(N(Y))^*
}$$
such that map~\eqref{eqn:isom-period-domain} is equivariant with respect to it. This establishes the isomorphism
$$
    \cM_\cpx(X)\cong\cM_\Kah(Y).
$$
The other isomorphism $\cM_\Kah(X)\cong\cM_\cpx(Y)$ is induced by the isometry $N(X)
\cong T(Y)$ in a similar way.
\end{proof}

Let $X$ be an abelian surface $X$ and consider the set of complexified K\"ahler classes
$$
    \cK(X)\colonequals\left\{
        \omega\in\NS(X)\otimes\bC
    \;\middle|\;
        \mathrm{Im}(\omega)\in\Amp(X)
    \right\}.
$$
Taking symplectic duals defines an involution
$$
    \iota\colon\cK(X)\longrightarrow\cK(X)
    :\omega\longmapsto\frac{\omega}{-\frac{1}{2}\omega^2}
$$
that reverses the K\"ahler volume. On the other hand, the exponential map defines an isomorphism
$$
    \cK(X)\longrightarrow\cQ^+(X)
    :\omega\longmapsto\left[
        \exp(\omega) = 1 + \omega + \frac{1}{2}\omega^2
    \right].
$$
Via this identification, the involution $\iota$ induces an involution on $\cQ^+(X)$. Let us prove that this descends to an involution on the stringy K\"ahler moduli space
$$
    \cM_\Kah(X)\cong\cQ^+(X) \,/ \,\SO^+(N(X))^*
$$
and that, when $X$ admits a mirror partner $Y$, the induced involution on $\cM_\cpx(Y)$ pairs $\NS(Y)$-polarized abelian surfaces whose underlying abelian surfaces are dual to each other.

\begin{lem}
\label{lem:double-cover}
The involution $\iota$ induces an involution on $\cM_\Kah(X)$ given by the action of an involution in
$
    \uO^+(N(X))^*
    \setminus
    \SO^+(N(X))^*,
$
which in turn corresponds to the covering involution of the double cover
$$\xymatrix{
    \cM_\Kah(X)\cong
    \cQ^+(X) \,/ \,\SO^+(N(X))^*
    \ar[r]^-{2:1}
    & \cQ^+(X) \,/ \,\uO^+(N(X))^*.
}$$
\end{lem}
\begin{proof}
Let us denote a vector in $N(X)$ as $(r, D, s)$ with respect to the decomposition
$$
    N(X) = H^0(X,\bZ)\oplus \NS(X)\oplus H^4(X,\bZ)
$$
With this notation, the involution on $\cQ^+(X)$ induced by $\iota$ maps
$$
    [\exp(\omega)]
    = \left[\left(
        1, \omega, \frac{1}{2}\omega^2
    \right)\right]
    \in\bP(N(X)\otimes\bC)
$$
to the class
$$
    \left[
        \exp\left(
            \frac{\omega}{-\frac{1}{2}\omega^2}
        \right)
    \right] = \left[\left(
        1, \frac{\omega}{-\frac{1}{2}\omega^2}, \frac{2}{\omega^2}
    \right)\right] = \left[\left(
        -\frac{1}{2}\omega^2, \omega, -1
    \right)\right].
$$
This involution corresponds to the isometry on $N(X)$ that
\begin{itemize}
\item acts trivially on $\NS(X)$, and
\item swaps the two summands of $H^0(X, \bZ)\oplus H^4(X, \bZ) = U(-1)$ followed by $-\mathrm{id}$.
\end{itemize}
This isometry has determinant~$-1$, so it lies in
$
    \uO^+(N(X))^*
    \setminus
    \SO^+(N(X))^*.
$
Thus the statement follows.
\end{proof}

\begin{proof}[Proof of Theorem~\ref{thm:symp-dual}]
Let $Y$ be a mirror partner of $X$. Fix an isometry $T(Y)\cong N(X)$ and consider the induced isomorphism
$$
    \mathscr{M}\colon\cM_\cpx(Y)\longrightarrow\cM_{\Kah}(X).
$$
By Lemma~\ref{lem:double-cover}, the involution $\iota$ induces an involution
$$
    \sD\colon\cM_\Kah(X)\longrightarrow\cM_\Kah(X)
$$
such that the composition
$$
    \mathscr{M}^{-1}\circ\sD\circ\mathscr{M}
    \colon\cM_\cpx(Y)\longrightarrow\cM_\cpx(Y)
$$
is given by the action of an involution in 
$
    \uO^+(T(Y))^*
    \setminus
    \SO^+(T(Y))^*.
$
As a consequence of \cite{Shioda}*{Theorem~2}, such an isometry pairs $\NS(Y)$-polarized abelian surfaces whose underlying abelian surfaces are dual to each other. This completes the proof.
\end{proof}

For an $(M, h)$-polarized abelian surface $(X, \mu)$, it is not a priori clear whether the dual abelian surface
$$
    \widehat{X}\colonequals H^1(X,\cO_X)/H^1(X,\bZ)
$$
admits a marking $\widehat{\mu}$ such that $(\widehat{X}, \widehat{\mu})$ is again $(M, h)$-polarized. Theorem~\ref{thm:symp-dual}~\ref{symp-dual:mirror-dual} implies that this is indeed the case when $(X, \mu)\in\cM_\cpx(Y)$ for some abelian surface $Y$ admitting a mirror partner. In what follows, we we prove that the same conclusion holds whenever $M$ satisfies Condition~\hyperref[deflem:mirror]{$\diamondsuit$}.

\begin{cor}
\label{cor:dual-invol-on-M_cpx}
Let $M\subseteq\Lambda$ be a primitive sublattice of signature $(1,r-1)$ satisfying Condition~\hyperref[deflem:mirror]{$\diamondsuit$}, and fix a vector $h\in M\otimes\bR$ with $h^2>0$. Then $\cM_\cpx(M)$ admits an involution that pairs $(M, h)$-polarized abelian surfaces whose underlying abelian surfaces are dual to each other.
\end{cor}

We present two proofs. The first is short but non-constructive. The second is constructive and is based on Shioda’s computation in \cite{Shioda}.

\begin{proof}[Proof~I]
For an abelian surface $Y$ underlying a very general member of $\cM_\cpx(M)$, we have that $M\cong\NS(Y)$, so the space $\cM_\cpx(M)$ can be naturally identified with $\cM_\cpx(Y)$. By hypothesis, $M$ satisfies Condition~\hyperref[deflem:mirror]{$\diamondsuit$}, so there exists a primitive sublattice $N\subseteq\Lambda$, satisfying the same condition, such that $M^\perp\cong N\oplus U$. Let $X$ be an abelian surface underlying a very general member of $\cM_\cpx(N)$. Then $N\cong\NS(X)$, which implies
$$
    N(X)\cong N\oplus U\cong M^\perp\cong T(Y).
$$
Therefore, $X$ is a mirror partner of $Y$. Thus $\cM_\cpx(Y)\cong\cM_\Kah(X)$ by Theorem~\ref{thm:mirror-moduli}. The statement then follows by applying Theorem~\ref{thm:symp-dual}~\ref{symp-dual:mirror-dual}.
\end{proof}

\begin{proof}[Proof~II]
Let $(X, \mu)$ be an $(M, h)$-polarized abelian surface in $\cM_\cpx(M)$, and let $\widehat{X}$ be the abelian surface dual to $X$. Our goal is to find a canonical marking
$$
    \widehat{\mu}\colon H^1(\widehat{H}, \bZ)
    \longrightarrow L
$$
such that $(\widehat{X}, \widehat{\mu})$ is $(M, h)$-polarized. First, the hypothesis asserts that $M$ admits a primitive embedding into the sum of two copies of $U$ in $\Lambda\cong U^{\oplus 3}$. Consider the two isometries:
\begin{itemize}
\item Let $\phi\in\uO^+(\Lambda)$ be the isometry that acts trivially on the $U^{\oplus2}$ containing $M$ and swaps the standard basis vectors of the remaining copy of $U$.
\item Let $\psi\in\uO^+(\Lambda)$ be the isometry that acts as $-\mathrm{id}$ on the $U^{\oplus2}$ containing $M$ and acts trivially on the remaining copy of $U$.
\end{itemize}
Note that $\det(\phi) = -1$ and $\det(\psi) = 1$. Moreover, $\psi = \wedge^2\nu$ where $\nu$ is an admissible automorphism of $L$ that relabels the basis elements.

By \cite{Shioda}*{Lemma~3}, there exists a canonical isomorphism
$$
    \alpha\colon H^2(\widehat{X},\bZ)
    \longrightarrow H^2(X,\bZ)
$$
that preserves the periods and satisfies $\det(\alpha) = -1$. Consider the composition
$$\xymatrix{
    H^2(\widehat{X}, \bZ)\ar[r]^-\alpha
    & H^2(X, \bZ)\ar[r]^-{\wedge^2\mu}
    & \Lambda.
}$$
Since $(X,\mu)$ is $M$-polarized and $\alpha$ preserves periods, the preimage of $M$ under this composition lies in $\NS(\widehat{X})$. By \cite{Shioda}*{Lemma~1}, the composition
\begin{equation}
\label{eqn:constr-pol-on-dual}
\xymatrix{
    H^2(\widehat{X}, \bZ)\ar[r]^-\alpha
    & H^2(X, \bZ)\ar[r]^-{\wedge^2\mu}
    & \Lambda\ar[rr]^-{\text{either}\;\phi\;\text{or}\;-\phi}
    && \Lambda
}
\end{equation}
is equal to $\wedge^2\widehat{\mu}$ for a unique admissible marking
$$
    \widehat{\mu}\colon H^1(\widehat{H}, \bZ)
    \longrightarrow L.
$$
Note that $(\wedge^2\widehat{\mu})^{-1}(M)\subseteq\NS(\widehat{X})$.
\begin{itemize}
\item If $(\wedge^2\widehat{\mu})^{-1}(h)$ is ample, then $(\widehat{X}, \widehat{\mu})$ is $(M, h)$-polarized.
\item If $(\wedge^2\widehat{\mu})^{-1}(h)$ is anti-ample, we further compose \eqref{eqn:constr-pol-on-dual} with $\psi$. The resulting composition equals $\wedge^2\widehat{\mu}'$ with $\widehat{\mu}' = \nu\circ\widehat{\mu}$. Then $(\widehat{X}, \widehat{\mu}')$ is $(M, h)$-polarized.
\end{itemize}
This completes the proof.
\end{proof}

\begin{rmk}
The isometry on $N(X) = H^0(X, \bZ)\oplus\NS(X)\oplus H^4(X, \bZ)$ that acts trivially on $\NS(X)$ and swaps the two summands of $H^0(X, \bZ)\oplus H^4(X, \bZ)$ followed by $-\mathrm{id}$ also appears in the context of K3~surfaces. In that setting, it corresponds to the cohomological action of the spherical twist along the structure sheaf. In contrast, abelian surfaces have no spherical object, and such an isometry cannot be realized as the cohomological action of any autoequivalence \cite{Yos09}*{Proposition~4.5}.
\end{rmk}

\subsection{Existence of mirror partners}
\label{subsec:crit_mirror-pair}

We now show that an abelian surface $X$ with Picard number $\rho(X)$ admits a mirror partner if and only if
\begin{itemize}
    \item $\rho(X)\leq2$, or
    \item $\rho(X)=3$ and $\NS(X)\cong\bZ(-2n)\oplus U$ for some positive integer~$n$.
\end{itemize} 

\begin{proof}[Proof of Theorem~\ref{thm:crit_mirror-partner}]
Suppose $Y$ is a mirror partner of $X$. Then
$$
    \NS(X)^{\perp\Lambda}
    \cong T(X)
    \cong N(Y)
    \cong \NS(Y)\oplus U.
$$
That is, the lattices $M\colonequals\NS(X)$ and $N\colonequals\NS(Y)$ satisfy Lemma~\ref{deflem:mirror}~\ref{mirror:N-in-Mperp}. The conclusion then follows from Lemma~\ref{deflem:mirror}~\ref{mirror:rank}. Conversely, if $M = \NS(X)$ satisfies Lemma~\ref{deflem:mirror}~\ref{mirror:rank}, then there exists a primitive sublattice $N\subseteq\Lambda$ such that $N^{\perp\Lambda}\cong M\oplus U$ by Lemma~\ref{deflem:mirror}~\ref{mirror:M-in-Nperp}. Note that $N$ satisfies Condition~\hyperref[deflem:mirror]{$\diamondsuit$}. Take a very general $Y\in\cM_\cpx(N)$. Then $\NS(Y)\cong N$, and
$$
    T(Y)\cong N^{\perp\Lambda}
    \cong M\oplus U
    \cong N(X).
$$
Hence $Y$ is a mirror partner of $X$.
\end{proof}

For an abelian surface $X$ of Picard number~$1$, we have $\NS(X)\cong\bZ(2n)$ for some positive integer~$n$. Following the computation in the proof of \cite{Dol96}*{Theorem~7.1}, one can verify that its stringy K\"ahler moduli space is
$$
    \cM_\Kah(X)\cong\cH/\Gamma_0(n),
$$
the quotient of the upper half-plane by the Hecke congruence subgroup
$$
    \Gamma_0(n) = \left\{
        \begin{pmatrix}
            a & b \\
            c & d
        \end{pmatrix}\in\PSL_2(\bZ)
    \;\middle|\;
        c\equiv 0\pmod{n}
    \right\}.
$$

In the setting of K3~surfaces, the corresponding moduli space is constructed by first removing from $\cQ^+(X)$ the hyperplane sections orthogonal to $(-2)$-vectors. That is, one considers
$$
    \cQ^+_0(X) = \cQ^+(X)
    \;\Big\backslash\;
    \bigcup_{\delta\in N(X),\,\delta^2 = -2}
    \delta^\perp,
$$
and then takes the quotient by $\uO^+(N(X))^*$ instead of $\SO^+(N(X))^*$. In the case of Picard number~$1$, this yields a quotient $\cH^0/\Gamma_0^+(n)$, where $\cH^0\subseteq\cH$ is a suitable open subset, and the subgroup $\Gamma_0^+(n)\subseteq\PSL(2,\bR)$ is the extension of $\Gamma_0(n)$ by the Fricke involution
$$
    w_n = \begin{pmatrix}
        0 & -1/\sqrt{n} \\
        \sqrt{n} & 0
    \end{pmatrix}.
$$

Returning to the setting of abelian surfaces, a mirror partner $Y$ of an abelian surface of Picard number~$1$ has $\NS(Y)\cong U\oplus\bZ(-2n)$. In particular, $\NS(Y)$ contains a hyperbolic plane $U$, so $Y$ is isomorphic to a product of elliptic curves $E \times E'$. Moreover, the summand $\bZ(-2n)\subseteq\NS(Y)$ implies the existence of a cyclic isogeny $\phi\colon E\to E'$ of degree~$n$; see, for example, the proof of \cite{Ma11}*{Proposition 4.1}. It is well known that $\cH/\Gamma_0(n)$ is the coarse moduli space of isomorphism classes of pairs $(E,A)$, where $E$ is an elliptic curve and $A\subseteq E$ is a cyclic subgroup of order~$n$ \cite{DS05_ModForm}*{Theorem~1.5.1}. This is compatible with the isomorphisms
$$
    \cM_\cpx(Y)\cong\cM_\Kah(X)\cong\cH/\Gamma_0(n).
$$

Mirror symmetry for abelian surfaces of Picard number~$2$ is more subtle than in the cases of Picard number~$1$ or~$3$. In contrast to those cases, the N\'eron--Severi lattice of a mirror partner is not, in general, uniquely determined by that of the original surface. In what follows, we first establish a lemma that characterizes mirror pairs in terms of discriminant forms. This will be used in Section~\ref{sec:self-mirror} to classify self-mirror abelian surfaces. We then give a criterion for deciding when two abelian surfaces have the same mirror.

\begin{lem}
\label{lem:mirror-pair}
Let $M, N\subseteq U^{\oplus 3}$ be primitive sublattices of signatures $(1, r-1)$, $(1, s-1)$, respectively, with $r + s = 4$, and let $(A_M, q_M)$, $(A_N, q_N)$ be their discriminant forms. Then the following conditions are equivalent:
\begin{enumerate}[label=\textup{(\arabic*)}]
    \item\label{mirror-pair:Mperp=N+U} $M^\perp\cong N\oplus U$.
    \item\label{mirror-pair:Nperp=M+U} $N^\perp\cong M\oplus U$.
    \item\label{mirror-pair:N-in-Mperp} There exists a primitive embedding $M\hookrightarrow U^{\oplus2}$ whose orthogonal complement is isomorphic to $N$.
    \item\label{mirror-pair:M-in-Nperp} There exists a primitive embedding $\,N\hookrightarrow U^{\oplus2}$ whose orthogonal complement is isomorphic to $M$.
    \item\label{mirror-pair:anti-isom} $A_M$ and $A_N$ are anti-isometric, that is, there exists an isomorphism $f\colon A_M\to A_N$ such that $-q_M = q_N\circ f$.
\end{enumerate}
In particular, two abelian surfaces $X$ and $Y$ are mirror partners if and only if $M\colonequals\NS(X)$ and $N\colonequals\NS(Y)$ satisfy any of the above conditions.
\end{lem}
\begin{proof}
Suppose $M^\perp\cong N\oplus U$. Since all primitive embeddings $U\hookrightarrow U^{\oplus 3}$ are equivalent under the action of $\uO(U^{\oplus 3})$, we may identify the $U$-summand of $M^\perp\cong N\oplus U$ with one of the $U$-summands of $U^{\oplus 3}$. With this identification, $M = (M^\perp)^\perp$ appears as a primitive sublattice of the remaining $U^{\oplus 2}$-summand, and with $N$ as its orthogonal complement. This proves \ref{mirror-pair:Mperp=N+U}$\implies$\ref{mirror-pair:N-in-Mperp}.

Conversely, by identifying the codomain of the given embedding $M\hookrightarrow U^{\oplus 2}$ with a $U^{\oplus 2}$-summand of $U^{\oplus 3}$, we obtain a primitive embedding $M\hookrightarrow U^{\oplus 3}$ whose orthogonal complement is isomorphic to $N\oplus U$. Hence \ref{mirror-pair:Mperp=N+U} $\Longleftrightarrow$ \ref{mirror-pair:N-in-Mperp}. The proof of \ref{mirror-pair:Nperp=M+U} $\Longleftrightarrow$ \ref{mirror-pair:M-in-Nperp} is similar.

By \cite{Nik79}*{Corollary~1.6.2}, two even lattices $M$ and $N$ occur as primitive and mutually orthogonal sublattices of an even unimodular lattice, which in our case is $U^{\oplus 2}$, if and only if there exists an isomorphism $f\colon A_M\to A_N$ satisfying $-q_M = q_N\circ f$. This gives the equivalences \ref{mirror-pair:N-in-Mperp} $\Longleftrightarrow$ \ref{mirror-pair:anti-isom} $\Longleftrightarrow$ \ref{mirror-pair:M-in-Nperp}.

The last statement then follows from the definition of mirror partners.
\end{proof}

\begin{prop}
\label{prop:stably-equiv}
Let $M_1$ and $M_2$ be even lattices of signature $(1,1)$ with discriminant forms $(A_{M_1}, q_{M_1})$ and $(A_{M_2}, q_{M_2})$. Then the following statements are equivalent:
\begin{enumerate}[label=\textup{(\arabic*)}]
    \item\label{stably-equiv:same-mirror}
    There exists an even lattice $N$ with primitive embeddings $\iota_i\colon N\hookrightarrow U^{\oplus 2}$, $i = 1,2$, such that $M_i\cong\iota_i(N)^\perp$.
    \item\label{stably-equiv:same-disc}
    There exists an isomorphism $f\colon A_{M_1}\rightarrow A_{M_2}$ such that $q_{M_1}=q_{M_2}\circ f$.
    \item\label{stably-equiv:plus-U}
    There exists an isomorphism $M_1\oplus U\cong M_2\oplus U$.
    \item\label{stably-equiv:stably-equiv}
    The lattices $M_1$ and $M_2$ are stably equivalent, that is, $M_1\oplus M_1'\cong M_2\oplus M_2'$ for some unimodular lattices $M_1'$ and $M_2'$.
\end{enumerate}
In particular, two abelian surfaces $X_1$ and $X_2$ of Picard number~$2$ share the same mirror partner if and only if $\NS(X_1)$ and $\NS(X_2)$ are stably equivalent.
\end{prop}
\begin{proof}
The implication \ref{stably-equiv:same-mirror}$\implies$\ref{stably-equiv:same-disc} follows from Lemma~\ref{lem:mirror-pair}. To prove the converse, choose any primitive embedding $M_1\hookrightarrow U^{\oplus 2}$, and let $N$ be its orthogonal complement. Then $A_{M_1}$ and $A_N$ are anti-isometric, and hence $A_{M_2}$ is also anti-isometric to $A_N$. By Lemma~\ref{lem:mirror-pair} again, the lattice $N$ admits a primitive embedding into $U^{\oplus 2}$ with orthogonal compliment isomorphic to $M_2$. This proves \ref{stably-equiv:same-mirror} $\Longleftrightarrow$ \ref{stably-equiv:same-disc}.

Under the situation of \ref{stably-equiv:same-mirror}, if we identify the codomain of $\iota_i$ with a $U^{\oplus 2}$-summand of $U^{\oplus 3}$, then the orthogonal complement of $N$ becomes $M_i\oplus U$. Since $N$ has signature $(1,1)$, the two embeddings of $N$ into $U^{\oplus 3}$ are equivalent under the action of $\uO(U^{\oplus})$. Applying such an action yields an isomorphism $M_1\oplus U\cong M_2\oplus U$. This shows \ref{stably-equiv:same-mirror}$\implies$\ref{stably-equiv:plus-U}.

The implication \ref{stably-equiv:plus-U}$\implies$\ref{stably-equiv:stably-equiv} is immediate, and the equivalence \ref{stably-equiv:stably-equiv} $\Longleftrightarrow$ \ref{stably-equiv:same-disc} follows from \cite{Nik79}*{Theorem~1.3.1}. Together, these implications establish the equivalence of all four statements.
\end{proof}

\begin{rmk}
Consider the set of equivalence classes of abelian surfaces of Picard number~$2$, where each class $[X]$ consists of those $X'$ with $\NS(X')$ stably equivalent to $\NS(X)$. By Proposition~\ref{prop:stably-equiv}, mirror symmetry defines an involution on this set. Note that if two even lattices $M_1, M_2\subseteq U^{\oplus 3}$ of signature~$(1,1)$ are stably equivalent, then their orthogonal complements are equivalent under the action of $\uO(U^{\oplus 3})$. By the construction of complex moduli, this yields a canonical isomorphism $\cM_\cpx(M_1)\cong\cM_\cpx(M_2)$. Hence, we have canonical isomorphisms $\cM_\cpx(X_1)\cong\cM_\cpx(X_2)$ for all $X_1, X_2$ in the same equivalence class.
\end{rmk}

\section{Self-mirror abelian surfaces}
\label{sec:self-mirror}

The goal of this section is to classify self-mirror abelian surfaces, namely those abelian surfaces $X$ satisfying
$
    T(X)\cong N(X),
$
in terms of the discriminant forms of their N\'eron--Severi lattices. Note that in this case $\cM_\cpx(X)\cong\cM_\Kah(X)$.

\subsection{Anti-automorphisms on discriminant forms}

Let $X$ be an abelian surface, and denote by $(A, q)$ the discriminant form of $\NS(X)$. By Lemma~\ref{lem:mirror-pair}, $X$ is self-mirror if and only if there exists an automorphism
$$
    f\colon A\longrightarrow A
    \qquad\text{such that}\qquad
    q = -q\circ f.
$$
We call such a map $f$ an \emph{anti-automorphism}. For each prime $p$, the restriction of the quadratic form $q$ to the $p$-Sylow subgroup $A_p\subseteq A$ takes the form
$$
    q_p\colon A_p\longrightarrow
    \begin{cases}
        2\bQ^{(p)}/2\bZ, & p \neq 2, \\
        \bQ^{(2)}/2\bZ, & p = 2,
    \end{cases}
$$
where $\bQ^{(p)}\subseteq\bQ$ is the subset of rational numbers whose denominators are powers of $p$.
Using the Chinese Remainder Theorem, one can check that $(A,q)$ admits an anti-automorphism if and only if each $(A_p, q_p)$ does. Hence, the problem reduces to checking the existence of anti-automorphisms locally at each prime.

Since $\NS(X)$ has rank two in our setting, each local component $(A_p, q_p)$ has one of the following forms by \cite{Nik79}*{Proposition~1.8.1}:
\begin{itemize}
    \item $p$ is odd, $A_p\cong\bZ/p^k\bZ$ for some integer $k$, and there exists an integer $\theta$ not divisible by $p$ such that
    $$
        q_p(\overline{a}) = \frac{2\theta a^2}{p^k}
        \qquad\text{for each}\qquad
        \overline{a}\in A_p.
    $$
    Here $a$ is an integer representing $\overline{a}\in\bZ/p^k\bZ$.
    \item $p$ is odd, $A_p\cong(\bZ/p^{k_1}\bZ)\oplus(\bZ/p^{k_2}\bZ)$ for some integers $k_1, k_2$, and there exist integers $\theta_1, \theta_2$ not divisible by $p$ such that
    $$
        q_p(\overline{a}_1, \overline{a}_2)
        = \frac{2\theta_1 a_1^2}{p^{k_1}}
            + \frac{2\theta_2 a_2^2}{p^{k_2}}
        \qquad\text{for each}\qquad
        (\overline{a}_1, \overline{a}_2)\in A_p.
    $$
    Here $a_1, a_2$ are integers representing $\overline{a}_1\in\bZ/p^{k_1}\bZ$ and $\overline{a}_2\in\bZ/p^{k_2}\bZ$.
    \item $p = 2$, $A_2\cong\bZ/2^k\bZ$ for some integer $k$, and there exists an odd integer $\theta$ such that
    $$
        q_2(\overline{a})
        = \frac{\theta a^2}{2^k}
        \qquad\text{for each}\qquad
        \overline{a}\in A_2.
    $$
    Here $a$ is an integer representing $\overline{a}\in\bZ/2^k\bZ$.
    \item $p=2$, $A_2\cong(\bZ/2^{k_1}\bZ)\oplus(\bZ/2^{k_2}\bZ)$ for some integers $k_1, k_2$, and there exist odd integers $\theta_1, \theta_2$ such that
    $$
        q_2(\overline{a}_1, \overline{a}_2)
        = \frac{\theta_1 a_1^2}{2^{k_1}}
            + \frac{\theta_2 a_2^2}{2^{k_2}}
        \qquad\text{for each}\qquad
        (\overline{a}_1, \overline{a}_2)\in A_2.
    $$
    Here $a_1, a_2$ are integers representing $\overline{a}_1\in\bZ/2^{k_1}\bZ$ and $\overline{a}_2\in\bZ/2^{k_2}\bZ$.
    \item $p=2$, $A_2\cong(\bZ/2^{k}\bZ)\oplus(\bZ/2^{k}\bZ)$ for some integer $k$, and there exists an odd integer $\theta$ such that
    $$
        q_2(\overline{a}_1,\overline{a}_2)
        = \frac{\theta a_1a_2}{2^{k-1}}
        \qquad\text{for each}\qquad
        (\overline{a}_1, \overline{a}_2)\in A_2.
    $$
    Here $a_1, a_2$ are integers representing $\overline{a}_1,\overline{a}_2\in\bZ/2^{k}\bZ$.
    \item $p=2$, $A_2\cong(\bZ/2^{k}\bZ)\oplus(\bZ/2^{k}\bZ)$ for some integer $k$, and there exists an odd integer $\theta$ such that
    $$
        q_2(\overline{a}_1,\overline{a}_2)
        = \frac{\theta(a_1^2+a_1a_2+a_2^2)}{2^{k-1}}
        \qquad\text{for each}\qquad
        (\overline{a}_1, \overline{a}_2)\in A_2.
    $$
    Here $a_1, a_2$ are integers representing $\overline{a}_1,\overline{a}_2\in\bZ/2^{k}\bZ$.
\end{itemize}

In the following, we determine precisely when a given $(A_p, q_p)$ listed above admits an anti-automorphism, starting from the case where $p$ is odd.

\begin{lem}
\label{lem:p-odd_cyclic}
Assume that $p$ is odd, $A_p\cong\bZ/p^k\bZ$ for some integer $k$, and there exists an integer~$\theta$ not divisible by $p$ such that
$$
    q_p(\overline{a})
    = \frac{2\theta a^2}{p^k}
    \in 2\bQ^{(p)}/2\bZ
    \qquad\text{for each}\qquad
    \overline{a}\in A_p.
$$
Then $(A_p, q_p)$ admits an anti-automorphism if and only if $p\equiv 1\pmod{4}$.
\end{lem}
\begin{proof}
An automorphism $f\colon\bZ/p^k\bZ\to\bZ/p^k\bZ$ is determined by $f(\overline{1}) = \overline{u}$, where $p\nmid u$. The condition $q_p = -q_p\circ f$ is then equivalent to $p^k\mid u^2+1$. It is a standard result in elementary number theory that such an $u$ exists if and only if $p\equiv1\pmod{4}$.
\end{proof}

\begin{lem}
\label{lem:p-odd_diag}
Assume that $p$ is odd, $A_p\cong(\bZ/p^{k_1}\bZ)\oplus(\bZ/p^{k_2}\bZ)$ for some integers $k_1, k_2$, and there exist integers $\theta_1, \theta_2$ not divisible by $p$ such that
$$
    q_p(\overline{a}_1, \overline{a}_2)
    = \frac{2\theta_1 a_1^2}{p^{k_1}}
        + \frac{2\theta_2 a_2^2}{p^{k_2}}
    \in 2\bQ^{(p)}/2\bZ
    \qquad\text{for each}\qquad
    (\overline{a}_1, \overline{a}_2)\in A_p.
$$
Then $(A_p, q_p)$ admits an anti-automorphism if and only if
\begin{itemize}
    \item $p\equiv 1\pmod{4}$, or
    \item $p\equiv 3\pmod{4}$ and $A_p\cong(\bZ/p^{k}\bZ)\oplus(\bZ/p^{k}\bZ)$.
\end{itemize}
\end{lem}
\begin{proof}
If $p\equiv 1\pmod{4}$, then there exists an integer $u$ such that $p^{\max\{k_1,k_2\}}\mid u^2+1$. The automorphism
$$
    f\colon(\bZ/p^{k_1}\bZ)\oplus(\bZ/p^{k_2}\bZ)
    \longrightarrow
    (\bZ/p^{k_1}\bZ)\oplus(\bZ/p^{k_2}\bZ), \quad
    \begin{cases}
        f(\overline{1}, \overline{0})
        = (\overline{u}, \overline{0}), \\
        f(\overline{0}, \overline{1})
        = (\overline{0}, \overline{u}),
    \end{cases}
$$
then satisfies $q_p = -q_p\circ f$. In the remaining part of the proof, we assume $p\equiv 3\pmod{4}$.

First, we show the existence of an anti-isomorphism $f$ forces $k_1 = k_2$. Assume instead that $k_1 < k_2$, and write $f(\overline{0}, \overline{1}) = (\overline{u}_1, \overline{u}_2)$. Evaluating the relation $-q_p = q_p\circ f$ at $(\overline{0}, \overline{1})$ gives
$$
    -\frac{2\theta_2}{p^{k_2}}
    = \frac{2\theta_1u_1^2}{p^{k_1}}
        + \frac{2\theta_2u_2^2}{p^{k_2}}
    \pmod{2\bZ}.
$$
This implies $p\mid u_2^2 + 1$, which is impossible when $p\equiv3\pmod{4}$. Hence $k_1 = k_2$.

Now we show that an anti-automorphism exists when $k \colonequals k_1 = k_2$. Consider the quadratic form
$$
    q_p\colon(\bZ/p^k\bZ)\oplus(\bZ/p^k\bZ)
    \longrightarrow
    2\bQ^{(p)}/2\bZ,
    \quad
        q_p(\overline{a}_1, \overline{a}_2)
    = \frac{2\theta_1 a_1^2}{p^k}
        + \frac{2\theta_2 a_2^2}{p^k},
    \quad p\nmid\theta_1\theta_2.
$$
and a group homomorphism
$$
    f\colon (\bZ/p^k\bZ)\oplus(\bZ/p^k\bZ)
    \longrightarrow
    (\bZ/p^k\bZ)\oplus(\bZ/p^k\bZ),
    \quad\begin{cases}
        f(\overline{1}, \overline{0})
        = (\overline{u}_{11}, \overline{u}_{12}), \\
        f(\overline{0}, \overline{1})
        = (\overline{u}_{21}, \overline{u}_{22}).
    \end{cases}
$$
For this map to be an anti-automorphism, it must satisfy $-q_p = q_p\circ f$. A straightforward computation shows that this holds if and only if the following conditions hold:
\begin{enumerate}[label=(\roman*)]
    \item\label{anti-aut:odd-p_(1,0)}
    $
        -q_p(\overline{1}, \overline{0})
        = q_p\circ f(\overline{1}, \overline{0}),
    $
    equivalently,
    $
        p^k\mid(u_{11}^2+1)\theta_1+u_{12}^2\theta_2,
    $
    \item\label{anti-aut:odd-p_(0,1)}
    $
        -q_p(\overline{0}, \overline{1})
        = q_p\circ f(\overline{0}, \overline{1}),
    $
    equivalently,
    $
        p^k\mid u_{21}^2\theta_1+(u_{22}^2+1)\theta_2,
    $
    \item\label{anti-aut:odd-p_(1,1)}
    $
        p^k\mid u_{11}u_{21}\theta_1+u_{12}u_{22}\theta_2.
    $
    \item\label{anti-aut:odd-p_aut}
    $f$ is an automorphism, equivalently,
    $
        p\nmid u_{11}u_{22} - u_{12}u_{21}.
    $
\end{enumerate}
It remains to find integers $u_{ij}$ satisfying these conditions. We proceed by induction.

\medskip
\noindent\textbf{Base case $\boldsymbol{k=1}$.}
Consider the sets
$$
    A = \left\{
        \overline{m^2}\in\bZ/p\bZ
    \;\middle|\;
        m\in\bZ
    \right\},
    \qquad
    B = \left\{
        \overline{n^2\theta_1\theta_2}\in\bZ/p\bZ
    \;\middle|\;
        n\in\bZ
    \right\}.
$$
Since $p\nmid\theta_1\theta_2$, we have $|A| = |B| = \frac{p+1}{2}$. By the Cauchy--Davenport theorem, their sum
$$
    A + B = \left\{
        \overline{m^2+n^2\theta_1\theta_2}\in\bZ/p\bZ
    \;\middle|\;
        m,n\in\bZ
    \right\}
$$
satisfies
$
    |A + B| \geq \min\left\{
        p,\, |A| + |B| - 1
    \right\} = p.
$
Hence $A + B = \bZ/p\bZ$. In particular, there exists integers $m$ and $n$ such that
\begin{equation}
\label{eqn:div-cond-mn}
    p\mid m^2 + n^2\theta_1\theta_2 + 1.
\end{equation}
Choose
$$
    u_{11} = u_{22} = m, \quad
    u_{12} = n\theta_1, \quad
    u_{21} = -n\theta_2.
$$
Then they satisfy Conditions~\ref{anti-aut:odd-p_(1,0)}--\ref{anti-aut:odd-p_aut} with $k=1$.

Note that $p\nmid n$. Otherwise, we would have $p\mid m^2 + 1$, which contradicts the assumption that $p\equiv 3\pmod{4}$. Thus, neither $u_{12}$ nor $u_{21}$ is divisible by $p$.

\medskip
\noindent\textbf{Inductive step.}
Assume there exist integers $u_{ij}$ satisfying \ref{anti-aut:odd-p_(1,0)}--\ref{anti-aut:odd-p_aut} for some $k\geq1$, with 
$$
    u_{11}\equiv u_{22}\equiv m\pmod{p}, \qquad
    u_{12}\equiv n\theta_1\pmod{p}, \qquad
    u_{21}\equiv -n\theta_2\pmod{p},
$$
where $m, n$ are integers chosen earlier so as to satisfy \eqref{eqn:div-cond-mn}. Write
$$
    u_{11}' = u_{11} + p^{k}a, \qquad
    u_{12}' = u_{12} + p^{k}b, \qquad
    u_{21}' = u_{21} + p^{k}c, \qquad
    u_{22}' = u_{22}
$$
where $a$, $b$, $c$ are integers to be determined.
\begin{itemize}
\item For Condition~\ref{anti-aut:odd-p_(1,0)}, we have
\begin{align*}
    &(u_{11}'^2 + 1)\theta_1 + u_{12}'^2\theta_2 \\
    &= (
        (u_{11}^2 + 1)\theta_1 + u_{12}^2\theta_2
    ) + 2p^k(
        u_{11}a\theta_1 + u_{12}b\theta_2
    ) + p^{2k}(
        a^2\theta_1 + b^2\theta_2
    ).
\end{align*}
By the induction hypothesis, the first term vanishes modulo $p^{k}$. The last term vanishes modulo $p^{k+1}$. Thus, to ensure the condition holds, namely
$$
    p^{k+1}\mid (u_{11}'^2 + 1)\theta_1 + u_{12}'^2\theta_2,
$$
it suffices to show that for any $\gamma_1\in\bZ$, there exist $a,b\in\bZ$ such that
\begin{equation}
\label{eqn:gamma_1}
    u_{11}a\theta_1 + u_{12}b\theta_2
    \equiv (ma + nb\theta_2)\theta_1
    \equiv\gamma_1\pmod{p}.
\end{equation}

\item Similarly, to ensure Condition~\ref{anti-aut:odd-p_(0,1)} holds, it suffices to show that for any $\gamma_2\in\bZ$, there exists $c\in\bZ$ such that
\begin{equation}
\label{eqn:gamma_2}
    u_{21}c\theta_1
    \equiv -nc\theta_1\theta_2
    \equiv\gamma_2\pmod{p}.
\end{equation}

\item For Condition~\ref{anti-aut:odd-p_(1,1)}, we have
\begin{align*}
    & u_{11}'u_{21}'\theta_1 + u_{12}'u_{22}'\theta_2 \\
    &= (u_{11} + p^{k}a)(u_{21} + p^{k}c)\theta_1
        + (u_{12} + p^{k}b)u_{22}\theta_2 \\
    &= (u_{11}u_{21}\theta_1 
            + u_{12}u_{22}\theta_2
        ) + p^k(
            u_{11}c\theta_1
            + u_{21}a\theta_1
            + u_{22}b\theta_2
        ) + p^{2k}ac\theta_1.
\end{align*}
Hence, for any $\gamma_3\in\bZ$, we need $a$, $b$, $c$ to satisfy
\begin{equation}
\label{eqn:gamma_3}
    u_{11}c\theta_1
        + u_{21}a\theta_1
        + u_{22}b\theta_2
    \equiv mc\theta_1
        - na\theta_1\theta_2
        + mb\theta_2
    \equiv \gamma_3\pmod{p}.
\end{equation}
\end{itemize}

The following choice satisfies \eqref{eqn:gamma_1}--\eqref{eqn:gamma_3}:
$$
    a = -m\theta_1^{-1}\gamma_1
        + m\theta_2^{-1}\gamma_2
        + n\gamma_3, \quad
    b = (\theta_1^{-1}\gamma_1 - ma)(n\theta_2)^{-1}, \quad
    c = -(n\theta_1\theta_2)^{-1}\gamma_2.
$$
The verification of \eqref{eqn:gamma_1} and \eqref{eqn:gamma_2} is straightforward, while \eqref{eqn:gamma_3} follows from the divisibility condition \eqref{eqn:div-cond-mn}. Finally, by construction,
$$
    u_{11}'u_{22}' - u_{12}'u_{21}'
    \equiv  u_{11}u_{22} - u_{12}u_{21}
    \not\equiv 0\pmod{p},
$$
so Condition~\ref{anti-aut:odd-p_aut} holds. This completes the proof.
\end{proof}

\subsection{Anti-automorphisms on 2-Sylow summands}

In the following, we classify the conditions under which the local component $(A_2, q_2)$ admits an anti-automorphism.

\begin{lem}
\label{lem:p=2_cyclic}
Assume that $A_2\cong\bZ/2^k\bZ$ for some integer $k$ and there exists an odd integer~$\theta$ such that
$$
    q_2(\overline{a})
    = \frac{\theta a^2}{2^k}
    \in\bQ^{(2)}/2\bZ
    \qquad\text{for each}\qquad
    \overline{a}\in A_2.
$$
Then $(A_2, q_2)$ has no anti-automorphism.
\end{lem}
\begin{proof}
In this case, an anti-automorphism exists if and only if there exists an odd integer~$u$ such that
$$
    -\frac{\theta}{2^k}
    = \frac{\theta u^2}{2^k}\pmod{2\bZ},
    \qquad\text{or equivalently,}\qquad
    2^{k+1}\mid u^2+1.
$$
This is impossible since $u^2\equiv1\pmod{4}$ for all odd $u$.
Hence, no anti-automorphism is allowed in this case.
\end{proof}

Next, we treat the case where $A_2$ is a product and  the form $q_2$ is diagonalized.

\begin{lem}
\label{lem:p=2_diag_k2=k1}
Assume that $A_2\cong(\bZ/2^{k}\bZ)\oplus(\bZ/2^{k}\bZ)$ for some integer $k$ and there exist odd integers $\theta_1, \theta_2$ such that
$$
    q_2(\overline{a}_1, \overline{a}_2)
    = \frac{\theta_1 a_1^2}{2^{k}}
        + \frac{\theta_2 a_2^2}{2^{k}}
    \in\bQ^{(2)}/2\bZ
    \qquad\text{for each}\qquad
    (\overline{a}_1, \overline{a}_2)\in A_2.
$$
Then $(A_2, q_2)$ admits an anti-automorphism if and only if $\theta_1 + \theta_2\equiv 0\pmod{4}$.
\end{lem}
\begin{proof}
Consider a group homomorphism $f\colon A_2\to A_2$ with
$$
    f(\overline{1}, \overline{0}) = (\overline{u}_{11}, \overline{u}_{12}), \qquad
    f(\overline{0}, \overline{1}) = (\overline{u}_{21}, \overline{u}_{22}).
$$
A straightforward computation shows that $f$ is an anti-automorphism if and only if the following conditions hold:
\begin{enumerate}[label=(\roman*)]
    \item\label{A2_diag_k2=k1_(1,0)}
    $2^{k+1}\mid(u_{11}^2 + 1)\theta_1 + u_{12}^2\theta_2$,
    \item\label{A2_diag_k2=k1_(0,1)}
    $2^{k+1}\mid u_{21}^2\theta_1 + (u_{22}^2 + 1)\theta_2$.
    \item\label{A2_diag_k2=k1_(1,1)}
    $2^k\mid u_{11}u_{21}\theta_1 + u_{12}u_{22}\theta_2$.
    \item\label{A2_diag_k2=k1_aut}
    $2\nmid u_{11}u_{22} - u_{12}u_{21}$.
\end{enumerate}
We claim that these conditions imply $\theta_1 + \theta_2\equiv 0\pmod{4}$. Note that Condition~\ref{A2_diag_k2=k1_aut} implies
$$
    (u_{11}^2, u_{12}^2)\equiv (1, 0), (0, 1),\text{ or }(1, 1)\pmod{4}.
$$
If $\theta_1 + \theta_2\not\equiv 0\pmod{4}$, then $\theta_1\equiv\theta_2\equiv \pm 1\pmod{4}$. 
This implies that $(u_{11}^2+1)\theta_1+u_{12}^2\theta_2$ is not divisible by $4$, a contradiction. Hence $\theta_1 + \theta_2\equiv 0\pmod{4}$.

We now prove that if $\theta_1 + \theta_2\equiv 0\pmod{4}$, then there exist integers $u_{ij}$ satisfying Conditions~\ref{A2_diag_k2=k1_(1,0)}--\ref{A2_diag_k2=k1_aut}.

\medskip
\noindent\textbf{The case $\boldsymbol{k = 1}$.}
Define $u_{11} = u_{22} = 0$ and $u_{12} = u_{21} = 1$. Then Conditions~\ref{A2_diag_k2=k1_(1,0)} and \ref{A2_diag_k2=k1_(0,1)} hold since $\theta_1 + \theta_2\equiv 0\pmod{4}$. Condition~\ref{A2_diag_k2=k1_(1,1)} holds because $u_{11}u_{21}\theta_1+u_{12}u_{22}\theta_2 = 0$, and Condition~\ref{A2_diag_k2=k1_aut} holds since $u_{11}u_{22} - u_{12}u_{21} = -1$.

\medskip
\noindent\textbf{The case $\boldsymbol{k = 2}$.}
Define
$$
    u_{11} = u_{22} = \begin{cases}
        0 & \text{if}\quad
            \theta_1 + \theta_2\equiv 0\pmod{8}, \\
        2 & \text{if}\quad
            \theta_1 + \theta_2\equiv 4\pmod{8},
    \end{cases}
$$
and let $u_{12} = u_{21} = 1$ as before.
\begin{itemize}
\item If $\theta_1 + \theta_2\equiv 0\pmod{8}$, then
\begin{gather*}
    (u_{11}^2 + 1)\theta_1 + u_{12}^2\theta_2
    = \theta_1 + \theta_2
    \equiv 0\pmod{8}, \\
    u_{21}^2\theta_1 + (u_{22}^2 + 1)\theta_2
    = \theta_1 + \theta_2
    \equiv 0\pmod{8},
\end{gather*}
and $u_{11}u_{21}\theta_1 + u_{12}u_{22}\theta_2 = 0$.
\item If $\theta_1 + \theta_2\equiv 4\pmod{8}$, then
\begin{gather*}
    (u_{11}^2 + 1)\theta_1 + u_{12}^2\theta_2
    = 5\theta_1 + \theta_2
    \equiv 4\theta_1 + 4 \pmod{8}, \\
    u_{21}^2\theta_1 + (u_{22}^2 + 1)\theta_2
    = \theta_1 + 5\theta_2
    \equiv 4 + 4\theta_2 \pmod{8},
\end{gather*}
which are both $0\pmod{8}$ since $\theta_1$ and $\theta_2$ are odd. Moreover,
$$
    u_{11}u_{21}\theta_1 + u_{12}u_{22}\theta_2
    = 2\theta_1 + 2\theta_2\equiv 0\pmod{4}.
$$
\end{itemize}
Thus Conditions~\ref{A2_diag_k2=k1_(1,0)}--\ref{A2_diag_k2=k1_(1,1)} hold. Condition~\ref{A2_diag_k2=k1_aut} holds since $u_{11}u_{22}$ is even and $u_{12}u_{21} = 1$.

\medskip
\noindent\textbf{Inductive step.}
Assume that there exist even integers $u_{11}, u_{22}$ and odd integers $u_{12}, u_{21}$ satisfying Conditions~\ref{A2_diag_k2=k1_(1,0)}--\ref{A2_diag_k2=k1_aut} for some $k\geq 2$. Write
$$
    u_{11}' = u_{11} + 2^{k}a, \qquad
    u_{12}' = u_{12} + 2^{k}b, \qquad
    u_{21}' = u_{21} + 2^{k}c, \qquad
    u_{22}' = u_{22}
$$
where $a$, $b$, $c$ are integers to be determined.
\begin{itemize}
\item For Condition~\ref{A2_diag_k2=k1_(1,0)}, we have
\begin{align*}
    &(u_{11}'^2 + 1)\theta_1 + u_{12}'^2\theta_2 \\
    &= (
        (u_{11}^2 + 1)\theta_1 + u_{12}^2\theta_2
    ) + 2^{k+1}(
        u_{11}a\theta_1 + u_{12}b\theta_2
    ) + 2^{2k}(
        a^2\theta_1 + b^2\theta_2
    ).
\end{align*}
By the induction hypothesis, the first term vanishes modulo $2^{k+1}$. Note that $k\geq 2$ implies $2k\geq k+2$, so the last term vanishes modulo $2^{k+2}$. Therefore, to ensure
$$
    2^{k+2}\mid (u_{11}'^2 + 1)\theta_1 + u_{12}'^2\theta_2,
$$
it suffices to show that for any $\gamma_1\in\bZ$, there exist $a,b\in\bZ$ such that
\begin{equation}
\label{eqn:gamma1-mod2}
    u_{11}a\theta_1 + u_{12}b\theta_2
    \equiv\gamma_1\pmod{2}.
\end{equation}

\item Similarly, to ensure Condition~\ref{A2_diag_k2=k1_(0,1)} holds, it suffices to show that for any $\gamma_2\in\bZ$, there exists $c\in\bZ$ such that
\begin{equation}
\label{eqn:gamma2-mod2}
    u_{21}c\theta_1
    \equiv\gamma_2\pmod{2}.
\end{equation}

\item For Condition~\ref{A2_diag_k2=k1_(1,1)}, we have
\begin{align*}
    & u_{11}'u_{21}'\theta_1 + u_{12}'u_{22}'\theta_2 \\
    &= (u_{11} + 2^{k}a)(u_{21} + 2^{k}c)\theta_1
        + (u_{12} + 2^{k}b)u_{22}\theta_2 \\
    &= (u_{11}u_{21}\theta_1 + u_{12}u_{22}\theta_2) + 2^k(
            u_{11}c\theta_1
            + u_{21}a\theta_1
            + u_{22}b\theta_2
        ) + 2^{2k}ac\theta_1.
\end{align*}
Hence, for any $\gamma_3\in\bZ$, we need $a$, $b$, $c$ to satisfy
\begin{equation}
\label{eqn:gamma3-mod2}
    u_{11}c\theta_1
    + u_{21}a\theta_1
    + u_{22}b\theta_2
    \equiv \gamma_3\pmod{2}.
\end{equation}
\end{itemize}
Using the facts that $u_{11},u_{22}$ are even, and $u_{12},u_{21},\theta_1,\theta_2$ are odd, it is straightforward to verify that the choice
$$
    a = \gamma_3, \quad
    b = \gamma_1, \quad
    c = \gamma_2,
$$
satisfies \eqref{eqn:gamma1-mod2}--\eqref{eqn:gamma3-mod2}. Finally, by construction,
$$
    u_{11}'u_{22}' - u_{12}'u_{21}'
    \equiv  u_{11}u_{22} - u_{12}u_{21}
    \not\equiv 0\pmod{2},
$$
so Condition~\ref{A2_diag_k2=k1_aut} holds. This completes the proof.
\end{proof}

\begin{lem}
\label{lem:p=2_diag}
Assume that $A_2\cong(\bZ/2^{k_1}\bZ)\oplus(\bZ/2^{k_2}\bZ)$ for some integers $k_1, k_2$ and there exist odd integers $\theta_1, \theta_2$ such that
$$
    q_2(\overline{a}_1, \overline{a}_2)
    = \frac{\theta_1 a_1^2}{2^{k_1}}
        + \frac{\theta_2 a_2^2}{2^{k_2}}
    \in\bQ^{(2)}/2\bZ
    \qquad\text{for each}\qquad
    (\overline{a}_1, \overline{a}_2)\in A_2.
$$
Then $(A_2, q_2)$ admits an anti-automorphism if and only if one of the following holds:
\begin{itemize}
    \item $A_2\cong(\bZ/2^{k}\bZ)\oplus(\bZ/2^{k}\bZ)$ for some integer $k$ and $\theta_1 + \theta_2\equiv 0\pmod{4}$.
    \item $A_2\cong(\bZ/2^{k}\bZ)\oplus(\bZ/2^{k+1}\bZ)$ for some integer $k$ and $\theta_1 + \theta_2\equiv 0\pmod{4}$.
\end{itemize}
\end{lem}
\begin{proof}
Assume without loss of generality that $k_1\leq k_2$. We claim that $k_2\leq k_1 + 1$. Suppose, to the contrary, that $k_2\geq k_1 + 2$. Writing $f(\overline{0}, \overline{1}) = (\overline{u}_1, \overline{u}_2)$, the relation $-q_2 = q_2\circ f$ evaluated at $(\overline{0}, \overline{1})$ yields
$$
    - \frac{\theta_2}{2^{k_2}}
    = \frac{\theta_1u_1^2}{2^{k_1}}
        + \frac{\theta_2u_2^2}{2^{k_2}}
    \pmod{2\bZ},
$$
which implies $4\mid u_2^2 + 1$, a contradiction. Hence $k_2 = k_1$ or $k_2 = k_1 + 1$. The case $k_1 = k_2$ is treated in Lemmas~\ref{lem:p=2_diag_k2=k1}. We address the case $k_2 = k_1 + 1$ below.

Denote $k\colonequals k_1$. Then $A_2\cong(\bZ/2^{k}\bZ)\oplus(\bZ/2^{k+1}\bZ)$ and $q_2$ has the form
$$
    q_2(\overline{a}_1, \overline{a}_2)
    = \frac{\theta_1a_1^2}{2^k} + \frac{\theta_2a_2^2}{2^{k+1}}.
$$
Consider a group homomorphism $f\colon A_2\to A_2$ with
$$
    f(\overline{1}, \overline{0}) = (\overline{u}_{11}, \overline{u}_{12}), \qquad
    f(\overline{0}, \overline{1}) = (\overline{u}_{21}, \overline{u}_{22}).
$$
Note that $2\mid u_{12}$ since $(\overline{1}, \overline{0})$ has order $2^k$. Write
$$
    u_{12} = 2t_{12}.
$$
If $f$ is an anti-automorphism, then necessarily $-q_2 = q_2\circ f$, and a straightforward computation shows that this identity holds if and only if
\begin{enumerate}[label=(\roman*)]
    \item\label{A2_diag_k2=k1+1_(1,0)}
    $2^{k+1}\mid(u_{11}^2 + 1)\theta_1 + 2t_{12}^2\theta_2$,
    \item\label{A2_diag_k2=k1+1_(0,1)}
    $2^{k+2}\mid 2u_{21}^2\theta_1 + (u_{22}^2 + 1)\theta_2$,
    \item\label{A2_diag_k2=k1+1_(1,1)}
    $2^k\mid u_{11}u_{21}\theta_1 + t_{12}u_{22}\theta_2$.
\end{enumerate}
Condition~\ref{A2_diag_k2=k1+1_(1,0)} implies $u_{11}$ is odd. It follows that $(u_{11}^2 + 1)\theta_1\equiv 2\pmod{4}$, which implies $t_{12}$ is odd. Similarly, Condition~\ref{A2_diag_k2=k1+1_(0,1)} implies $u_{21}$ and $u_{22}$ are odd. Hence $u_{21}^2\equiv u_{22}^2\equiv 1\pmod{8}$, and by Condition~\ref{A2_diag_k2=k1+1_(0,1)} again, $\theta_1 + \theta_2\equiv 0\pmod{4}$.

Next, we show that if $u_{11}$, $t_{12}$, $u_{21}$, $u_{22}$ are odd, then the map $f$ is injective, and hence an automorphism. Indeed, if
$$
    f(\overline{a}_1, \overline{a}_2)
    = \overline{a}_1(\overline{u}_{11}, 2\overline{t}_{12})
        + \overline{a}_2(\overline{u}_{21}, \overline{u}_{22})
    = (\overline{0}, \overline{0}),
$$
then
$$
    2^k\mid a_1u_{11} + a_2u_{21}
    \qquad\text{and}\qquad
    2^{k+1}\mid 2a_1t_{12} + a_2u_{22}.
$$
Hence
$$
    2^k\mid u_{22}(a_1u_{11} + a_2u_{21})
        - u_{21}(2a_1t_{12} + a_2u_{22})
    = a_1(u_{11}u_{22} - 2t_{12}u_{21}).
$$
Since $u_{11}u_{22} - 2t_{12}u_{21}$ is odd, it follows that $2^k\mid a_1$. This in turn implies $2^{k+1}\mid a_2$ by the relation $2^{k+1}\mid 2a_1t_{12} + a_2u_{22}$ and the assumption that $u_{22}$ is odd. Therefore,
$$
    (\overline{a}_1, \overline{a}_2)
    = (\overline{0}, \overline{0})
    \in A_2\cong(\bZ/2^{k}\bZ)\oplus(\bZ/2^{k+1}\bZ),
$$
which proves that $f$ is injective.

It remains to show that if $\theta_1 + \theta_2\equiv 0\pmod{4}$, then there exist odd integers $u_{11}$, $t_{12}$, $u_{21}$, $u_{22}$ satisfying Conditions~\ref{A2_diag_k2=k1+1_(1,0)}--\ref{A2_diag_k2=k1+1_(1,1)} with $u_{12} = 2t_{12}$.

\medskip
\noindent\textbf{The case $\boldsymbol{k = 1}$.}
The choice $u_{11} = t_{12} = u_{21} = u_{22} = 1$ provides a solution for this case.

\medskip
\noindent\textbf{The case $\boldsymbol{k = 2}$.}
Take $u_{11} = t_{12} = u_{21} = 1$ and set
$$
    u_{22} = \begin{cases}
        1 & \text{if}\quad
            \theta_1 + \theta_2\equiv 0\pmod{8}, \\
        5 & \text{if}\quad
            \theta_1 + \theta_2\equiv 4\pmod{8}.
    \end{cases}
$$
Then Conditions~\ref{A2_diag_k2=k1+1_(1,0)}--\ref{A2_diag_k2=k1+1_(1,1)} are satisfied.

\medskip
\noindent\textbf{Inductive step.}
Suppose that there exist odd integers $u_{11}$, $t_{12}$, $u_{21}$, $u_{22}$ satisfying Conditions~\ref{A2_diag_k2=k1+1_(1,0)}--\ref{A2_diag_k2=k1+1_(1,1)} for some $k\geq 2$. Set
$$
    u_{11}' = u_{11} + 2^ka, \qquad
    t_{12}' = t_{12} + 2^kb, \qquad
    u_{21}' = u_{21}, \qquad
    u_{22}' = u_{22} + 2^{k+1}c,
$$
where $a$, $b$, $c$ are integers to be determined.
\begin{itemize}
\item For Condition~\ref{A2_diag_k2=k1+1_(1,0)}, we have
\begin{align*}
    &(u_{11}'^2 + 1)\theta_1 + 2t_{12}'^2\theta_2 \\
    &= ((u_{11}^2 + 1)\theta_1 + 2t_{12}^2\theta_2)
        + 2^{k+1}u_{11}a\theta_1
        + 2^{k+2}t_{12}b\theta_2
        + 2^{2k}(a^2\theta_1 + 2b^2\theta_2).
\end{align*}
By the induction hypothesis, the first term vanishes modulo $2^{k+1}$. Note that $k\geq 2$ implies $2k\geq k+2$, so the last term vanishes modulo $2^{k+2}$. Therefore, the condition
$$
    2^{k+2}\mid (u_{11}'^2 + 1)\theta_1 + 2t_{12}'^2\theta_2
$$
holds if for any $\gamma_1\in\bZ$,
$$
    u_{11}a\theta_1
    \equiv a
    \equiv \gamma_1\pmod{2}.
$$
Thus we can take $a = \gamma_1$.

\item By a similar computation, Condition~\ref{A2_diag_k2=k1+1_(0,1)} holds in the inductive step if for any $\gamma_2\in\bZ$,
$$
    u_{22}c\theta_2
    \equiv c
    \equiv \gamma_2\pmod{2}.
$$
Thus we can take $c = \gamma_2$.

\item For Condition~\ref{A2_diag_k2=k1+1_(1,1)}, we have
\begin{align*}
    & u_{11}'u_{21}'\theta_1 + t_{12}'u_{22}'\theta_2 \\
    &= (u_{11} + 2^{k}a)u_{21}\theta_1
        + (t_{12} + 2^{k}b)(u_{22} + 2^{k+1}c)\theta_2 \\
    &= (u_{11}u_{21}\theta_1 + t_{12}u_{22}\theta_2)
        + 2^k(u_{21}a\theta_1 + u_{22}b\theta_2)
        + 2^{k+1}t_{12}c\theta_2
        + 2^{2k+1}bc\theta_2.
\end{align*}
Hence, for any $\gamma_3\in\bZ$, we need $a$ and $b$ to satisfy
$$
    u_{21}a\theta_1 + u_{22}b\theta_2
    \equiv \gamma_1 + b
    \equiv \gamma_3\pmod{2}.
$$
Thus we can take $b = \gamma_3 - \gamma_1$.
\end{itemize}
This completes the proof.
\end{proof}

\begin{lem}
\label{lem:p=2_U-type}
Assume that $A_2\cong(\bZ/2^{k}\bZ)\oplus(\bZ/2^{k}\bZ)$ for some integer $k$ and there exists an odd integer $\theta$ such that
$$
    q_2(\overline{a}_1, \overline{a}_2)
    = \frac{\theta a_1a_2}{2^{k-1}}
    \qquad\text{for each}\qquad
    (\overline{a}_1, \overline{a}_2)\in A_2.
$$
Then $(A_2, q_2)$ admits an anti-automorphism.
\end{lem}
\begin{proof}
In this case, the map
$
    (\overline{a}_1, \overline{a}_2)
    \longmapsto
    (\overline{a}_1, -\overline{a}_2)
$
provides an anti-automorphism.
\end{proof}

\begin{lem}
\label{lem:p=2_V-type}
Assume that $A_2\cong(\bZ/2^{k}\bZ)\oplus(\bZ/2^{k}\bZ)$ for some integer $k$ and there exists an odd integer $\theta$ such that
$$
    q_2(\overline{a}_1, \overline{a}_2)
    = \frac{\theta(a_1^2 + a_1a_2 + a_2^2)}{2^{k-1}}
    \qquad\text{for each}\qquad
    (\overline{a}_1, \overline{a}_2)\in A_2.
$$
Then $(A_2, q_2)$ admits an anti-automorphism.
\end{lem}
\begin{proof}
Consider a group homomorphism $f\colon A_2\to A_2$ with
$$
    f(\overline{1}, \overline{0}) = (\overline{u}_{11}, \overline{u}_{12}), \qquad
    f(\overline{0}, \overline{1}) = (\overline{u}_{21}, \overline{u}_{22}).
$$
If $f$ is an anti-automorphism, then necessarily $-q_2 = q_2\circ f$, and a straightforward computation shows that this identity holds if and only if
\begin{enumerate}[label=(\roman*)]
    \item\label{A2_V-type_(1,0)}
    $2^k\mid u_{11}^2 + u_{11}u_{12} + u_{12}^2 + 1$,
    \item\label{A2_V-type_(0,1)}
    $2^k\mid u_{21}^2 + u_{21}u_{22} + u_{22}^2 + 1$,
    \item\label{A2_V-type_(1,1)}
    $
        2^k\mid 2(u_{11}u_{21} + u_{12}u_{22})
            + (u_{11}u_{22} + u_{12}u_{21}) + 1.
    $
\end{enumerate}

Condition~\ref{A2_V-type_(1,0)} implies that exactly one of $u_{11}$ and $u_{12}$ is odd. Similarly, Condition~\ref{A2_V-type_(0,1)} implies exactly one of $u_{21}$ and $u_{22}$ is odd. Moreover, Condition~\ref{A2_V-type_(1,1)} implies $u_{11}u_{22} + u_{12}u_{21}$ is odd. Consequently, one of the following two cases must occur:
\begin{itemize}
    \item $u_{11}$, $u_{22}$ are odd while $u_{12}$, $u_{21}$ are even, or
    \item $u_{12}$, $u_{21}$ are odd while $u_{11}$, $u_{22}$ are odd.
\end{itemize}
In either case,
$
    2\nmid u_{11}u_{22} - u_{12}u_{21},
$
and hence $f$ is an automorphism.

It remains to show there exist odd integers $u_{11}$, $u_{22}$ and even integers $u_{12}$, $u_{21}$ satisfying Conditions~\ref{A2_V-type_(1,0)}--\ref{A2_V-type_(1,1)}. We proceed by induction.

\medskip
\noindent\textbf{Base case $\boldsymbol{k = 1}$.}
Taking $u_{11} = u_{22} = 1$ and $u_{12} = u_{21} = 0$ settles this case.

\medskip
\noindent\textbf{Inductive step.}
Assume that there exist odd $u_{11}$, $u_{22}$ and even $u_{12}$, $u_{21}$ satisfying Conditions~\ref{A2_V-type_(1,0)}--\ref{A2_V-type_(1,1)} for some $k\geq 1$. Write
$$
    u_{11}' = u_{11} + 2^ka, \qquad
    u_{12}' = u_{12} + 2^kb, \qquad
    u_{21}' = u_{21} + 2^kc, \qquad
    u_{22}' = u_{22},
$$
where $a$, $b$, $c$ are integers to be determined.
\begin{itemize}
\item For Condition~\ref{A2_V-type_(1,0)}, we have
$$
    u_{11}'^2 + u_{11}'u_{12}' + u_{12}'^2 + 1
    = (u_{11}^2+u_{11}u_{12}+u_{12}^2 + 1)
        + 2^k(u_{11}b + u_{12}a)
        + 2^{k+1}(\,\dots).
$$
By the induction hypothesis, the first term vanishes modulo $2^{k}$, so the condition
$$
    2^{k+1}\mid u_{11}'^2 + u_{11}'u_{12}' + u_{12}'^2 + 1
$$
holds if for any $\gamma_1\in\bZ$,
$$
    u_{11}b + u_{12}a
    \equiv b
    \equiv \gamma_1\pmod{2}.
$$
Thus we can take $b = \gamma_1$.

\item Similarly, Condition~\ref{A2_V-type_(0,1)} holds if for any $\gamma_2\in\bZ$,
$$
    cu_{22} \equiv c\equiv \gamma_2\pmod{2}.
$$
Thus we can take $c = \gamma_2$.

\item For Condition~\ref{A2_V-type_(1,1)}, we have
\begin{align*}
    &2(u_{11}'u_{21}' + u_{12}'u_{22}')
        + (u_{11}'u_{22}' + u_{12}'u_{21}') + 1 \\
    &= (
            2(u_{11}u_{21} + u_{12}u_{22})
            + (u_{11}u_{22} + u_{12}u_{21}) + 1
        ) + 2^{k}(
            u_{22}a + u_{21}b + u_{12}c
        ) + 2^{k+1}(\,\dots).
\end{align*}
Hence, for any $\gamma_3\in\bZ$, we need $a$, $b$, $c$ to satisfy
$$
    u_{22}a + u_{21}b + u_{12}c
    \equiv a
    \equiv \gamma_3\pmod{2}.
$$
Thus we can take $a = \gamma_3$.
\end{itemize}
This completes the proof.
\end{proof}

\begin{proof}[Proof of Theorem~\ref{thm:self-mirror}]
The theorem follows from Lemmas~\ref{lem:p-odd_cyclic}, \ref{lem:p-odd_diag}, which characterize the existence of anti-automorphisms of $(A_p,q_p)$ for odd $p$, and Lemmas~\ref{lem:p=2_cyclic}, \ref{lem:p=2_diag}, \ref{lem:p=2_U-type}, \ref{lem:p=2_V-type}, which treat the case $p=2$.
\end{proof}

\subsection{Examples of self-mirror abelian surfaces}

In what follows, we present examples of self-mirror abelian surfaces and give a criterion for when a principally polarized abelian surface is self-mirror.

\begin{eg}
\label{eg:product_AbS}
Consider a product abelian surface $X = E_1\times E_2$, where $E_1$ and $E_2$ are non-isogenous elliptic curves. Then $\NS(X)\cong U$, which is generated by the classes of the fibers $E_1\times\{0\}$ and $\{0\}\times E_2$. This abelian surface is self-mirror, since $U$ admits a primitive embedding into $U^{\oplus2}$ with orthogonal complement also isomorphic to $U$. The stringy K\"ahler moduli space $\cM_\Kah(X)$ in this setting, along with its associated Weil--Petersson metric, was studied in \cite{FKY}*{Section~4.2}.
\end{eg}

Before introducing further examples, we provide numerical criteria for an abelian surface to be \emph{simple}, that is, not
isogenous to a product of elliptic curves.

\begin{lem}
\label{lem:crit_simple-AbS}
An abelian surface $X$ is simple if and only if there exists no nonzero divisor class $D\in\NS(X)$ with $D^2 = 0$. In the case of Picard number~$2$, the surface $X$ is simple if and only if $|\mathrm{disc}\NS(X)|$ is not a perfect square.
\end{lem}
\begin{proof}
An abelian surface $X$ admits a nonzero $D\in\NS(X)$ with $D^2 = 0$ if and only if there exists an elliptic curve $E\subseteq X$ \cite{Kani94}*{Proposition~2.3}. In this case, Poincar\'e's Reducibility Theorem asserts that $X$ is isogenous to $E\times F$, where $F\subseteq X$ is the \emph{complementary elliptic curve} of $E$ \cite{BLComplexAbelVar}*{Theorem~5.3.5 and p.~125}. This proves the first statement.

Now, assume that $X$ has Picard number~$2$. If $X$ is simple, then there exists a nonzero and primitive class $D\in\NS(X)$ with $D^2 = 0$. We can extend $D$ to a basis of $\NS(X)$, with Gram matrix
$$\begin{pmatrix}
    0 & b \\
    b & 2c
\end{pmatrix}$$
for some integers $b, c$. This shows that $|\mathrm{disc}\NS(X)| = b^2$. To prove the converse, consider the Gram matrix of $\NS(X)$ with respect to a basis $\{e_1, e_2\}$:
$$\begin{pmatrix}
    2a & b \\
    b & 2c
\end{pmatrix}.$$
If $|\mathrm{disc}\NS(X)|$ is a square, then $b^2 - 4ac = n^2$ for some integer~$n$. In this setting, the class
$$
    D = (n - b)e_1 + 2ae_2
$$
satisfies $D^2 = 0$. This proves the second statement.
\end{proof}

\begin{eg}
A slight generalization of Example~\ref{eg:product_AbS} is given by abelian surfaces $X$ with
$$
    \NS(X)\cong U(n)
    = \begin{pmatrix}
        0 & n \\
        n & 0
    \end{pmatrix}.
$$
By Lemma~\ref{lem:crit_simple-AbS}, any such abelian surface is non-simple. Moreover, $X$ is self-mirror, since its N\'eron--Severi lattice has discriminant group
$$
    A_{\NS(X)}\cong(\bZ/n\bZ)\oplus(\bZ/n\bZ)
$$
with discriminant form
$$
    q(\overline{a}_1, \overline{a}_2)
    = \frac{2a_1a_2}{n},
$$
which admits an anti-automorphism given by
$
    (\overline{a}_1, \overline{a}_2)
    \mapsto
    (\overline{a}_1, -\overline{a}_2).
$
\end{eg}

So far, we have seen examples of non-simple abelian surfaces that are self-mirror. In general, simplicity and the self-mirror property are independent, as illustrated by the following examples.

\begin{eg}
Let $X$ be an abelian surface with
$$
    \NS(X)\cong\begin{pmatrix}
        0 & 3 \\
        3 & 2
    \end{pmatrix}.
$$
Then $X$ is non-simple by Lemma~\ref{lem:crit_simple-AbS}. The discriminant group of the N\'eron--Severi lattice is
$$
    A_{\NS(X)}\cong\bZ/9\bZ,
$$
which is cyclic. Hence $X$ is not self-mirror by Theorem~\ref{thm:self-mirror}~\ref{self-mirror:odd-p}.
\end{eg}

\begin{eg}
Let $X$ be an abelian surface with 
$$
    \NS(X)\cong\begin{pmatrix}
        2 & 3 \\
        3 & 2
    \end{pmatrix}.
$$
Since $|\text{disc}\NS(X)| = 5$ is not a square, the surface $X$ is simple by Lemma~\ref{lem:crit_simple-AbS}. The discriminant group of the N\'eron--Severi lattice is
$$
    A_{\NS(X)}\cong\bZ/5\bZ.
$$
In this case, $X$ is self-mirror by Theorem~\ref{thm:self-mirror}.
\end{eg}

For a principally polarized abelian surface $X$ of Picard number~$2$, the self-mirror criterion of Theorem~\ref{thm:self-mirror} reduce to a numerical condition on $|\text{disc}\NS(X)|$. More specifically, we prove that $X$ is self-mirror if and only if the discriminant of $\NS(X)$ is divisible by neither $16$ nor any prime $p\equiv3 \pmod{4}$.

\begin{proof}[Proof of Corollary~\ref{cor:self-mirror_prinAbS}]
Recall that an abelian surface admits a principal polarization if and only if there exists a class $\Theta\in\NS(X)$ with $\Theta^2 = 2$. The N\'eron--Severi lattice is therefore isomorphic to either
$$
    \begin{pmatrix}
        2 & 1 \\
        1 & -2n
    \end{pmatrix}
    \qquad\text{or}\qquad
    \begin{pmatrix}
        2 & 0 \\
        0 & -2n
    \end{pmatrix}
$$
for some integer~$n$. In the former case, the N\'eron-Severi lattice has discriminant group
$$
    A_{\NS(X)}\cong\bZ/(4n+1)\bZ,
$$
which is cyclic. By Theorem~\ref{thm:self-mirror}, the surface $X$ is self-mirror if and only if every prime $p$ dividing $4n + 1$ satisfies $p\equiv 1\pmod{4}$.

In the latter case, the N\'eron-Severi lattice has discriminant form
$$
    A\colonequals A_{\NS(X)}
    \cong(\bZ/2\bZ)\oplus(\bZ/2n\bZ), \qquad
    q(\overline{a}_1, \overline{a}_2)
    = \frac{a_1^2}{2} - \frac{a_2^2}{2n}.
$$
By Theorem~\ref{thm:self-mirror}, the surface $X$ is self-mirror if and only if the local components $(A_p, q_p)$ satisfy the following conditions:
\begin{itemize}
\item Suppose $p$ is odd. Then $A_p$ is cyclic, which implies $p\equiv 1\pmod{4}$. Thus $|\mathrm{disc}\NS(X)|$ has no prime factor $p\equiv 3\pmod{4}$.

\item Suppose $p = 2$ and write $2n = 2^ku$ for some odd $u$. Then
$$
    A_2\cong(\bZ/2\bZ)\oplus(\bZ/2^k\bZ), \qquad
    q_2(\overline{a}_1, \overline{a}_2)
    = \frac{a_1^2}{2} - \frac{ua_2^2}{2^k}.
$$
where $k\leq 2$ and $(1 - u)\equiv 0\pmod{4}$. Note that $k \leq 2$ is equivalent to $|\mathrm{disc}\NS(X)|$ not divisible by $16$. The relation $(1 - u)\equiv 0\pmod{4}$ holds once $|\mathrm{disc}\NS(X)|$ has no prime factor $p\equiv 3\pmod{4}$.
\end{itemize}
This completes the proof.
\end{proof}

\bigskip
\bibliography{MirrorAbS}
\bibliographystyle{alpha}

\ContactInfo
\end{document}